\documentclass[12pt,letterpaper,final,twoside,leqno]{amsart}
\usepackage{eucal, mathrsfs}
\usepackage{amsmath,amssymb,amsthm,
amscd,amsopn}
\usepackage{hyperref}
\usepackage{times}
\usepackage{xspace}
\usepackage{epsfig,epic,eepic,latexsym,color}
\usepackage[all]{xy}
\usepackage{paralist}
\usepackage{stmaryrd}
\usepackage{enumitem}
\usepackage{color}
\usepackage{soul}

\catcode`~=11 \def\UrlSpecials{\do\~{\kern -.15em\lower .7ex\hbox{~}\kern .04em}} \catcode`~=13

\newcommand{\urlwofont}[1]{\urlstyle{same}\url{#1}}

\newcommand{\widepagestyle}{
\voffset=0in
\hoffset=0in
\marginparwidth=0.7in
\oddsidemargin=0in
\evensidemargin=0in
\textwidth=6.5in
\textheight=8.5in
\topmargin=0in
\headheight=0in
\headsep=0.2in
\footskip=0.5in
}

\newcounter{are-there-sections}
\makeatletter

\renewcommand\subsection{
  \renewcommand{\sfdefault}{pag}
  \@startsection{subsection}%
  {2}{0pt}{-\baselineskip}{.2\baselineskip}{\raggedright
    \sffamily\itshape\small
  }}
\renewcommand\section{
  \renewcommand{\sfdefault}{phv}
  \@startsection{section} %
  {1}{0pt}{\baselineskip}{.2\baselineskip}{\centering
    \sffamily
    \scshape
}}
\newcounter{lastyear}
\addtocounter{lastyear}{-1}

\newcommand\noin{\noindent}

\newtheoremstyle{bozont}{3pt}{3pt}%
     {\itshape}
     {}
     {\bfseries}
     {.}
     {.5em}
     {\thmname{#1}\thmnumber{ #2}\thmnote{ \rm #3}}
\newtheoremstyle{bozont-sf}{3pt}{3pt}%
     {\itshape}
     {}
     {\sffamily}
     {.}
     {.5em}
     {\thmname{#1}\thmnumber{ #2}\thmnote{ \rm #3}}
\newtheoremstyle{bozont-sc}{3pt}{3pt}%
     {\itshape}
     {}
     {\scshape}
     {.}
     {.5em}
     {\thmname{#1}\thmnumber{ #2}\thmnote{ \rm #3}}
\newtheoremstyle{bozont-remark}{3pt}{3pt}%
     {}
     {}
     {\scshape}
     {}
     {.5em}
     {\thmname{#1}\thmnumber{ #2.}\thmnote{\ \  ( #3)}}
\newtheoremstyle{bozont-def}{3pt}{3pt}%
     {}
     {}
     {\bfseries}
     {.}
     {.5em}
     {\thmname{#1}\thmnumber{ #2}\thmnote{ \rm #3}}
\newtheoremstyle{bozont-reverse}{3pt}{3pt}%
     {\itshape}
     {}
     {\bfseries}
     {.}
     {.5em}
     {\thmnumber{#2.}\thmname{ #1}\thmnote{ \rm #3}}
\newtheoremstyle{bozont-reverse-sc}{3pt}{3pt}%
     {\itshape}
     {}
     {\scshape}
     {.}
     {.5em}
     {\thmnumber{#2.}\thmname{ #1}\thmnote{ \rm #3}}
\newtheoremstyle{bozont-reverse-sf}{3pt}{3pt}%
     {\itshape}
     {}
     {\sffamily}
     {.}
     {.5em}
     {\thmnumber{#2.}\thmname{ #1}\thmnote{ \rm #3}}
\newtheoremstyle{bozont-remark-reverse}{3pt}{3pt}%
     {}
     {}
     {\sc}
     {.}
     {.5em}
     {\thmnumber{#2.}\thmname{ #1}\thmnote{ \rm #3}}
\newtheoremstyle{bozont-def-reverse}{3pt}{3pt}%
     {}
     {}
     {\bfseries}
     {.}
     {.5em}
     {\thmnumber{#2.}\thmname{ #1}\thmnote{ \rm #3}}
\newtheoremstyle{bozont-def-newnum-reverse}{3pt}{3pt}%
     {}
     {}
     {\bfseries}
     {}
     {.5em}
     {\thmnumber{#2.}\thmname{ #1}\thmnote{ \rm #3}}
\theoremstyle{bozont}    
\ifnum \value{are-there-sections}=0 {%
  \newtheorem{proclaim}{Theorem}
} 
\else {%
  \newtheorem{proclaim}{Theorem}[section]
} 
\fi
\newtheorem{thm}[proclaim]{Theorem}

\newtheorem{cor}[proclaim]{Corollary} 
 
\newtheorem{lem}[proclaim]{Lemma} 
\newtheorem{prop}[proclaim]{Proposition}

\newtheorem{fact}[proclaim]{Fact}
\theoremstyle{bozont-sc}
\newtheorem{proclaim-special}[proclaim]{\specialthmname}

\theoremstyle{bozont-remark}

\newtheorem{rem}[proclaim]{Remark}

\newtheorem{const}[proclaim]{Construction}

\newtheorem*{SubHeading*}{\SubHeadingName}%
\newtheorem{SubHeading}[proclaim]{\SubHeadingName}
\newtheorem{sSubHeading}[equation]{\sSubHeadingName}
\newenvironment{demo-r}[1]{\def\SubHeadingName{#1}\begin{SubHeading-r}}
  {\end{SubHeading-r}}%
\newenvironment{subdemo-r}[1]{\def\sSubHeadingName{#1}\begin{sSubHeading-r}}
  {\end{sSubHeading-r}} %
\newenvironment{demo*}[1]{\def\SubHeadingName{#1}\begin{SubHeading*}}
  {\end{SubHeading*}}%

\newtheorem{question}[proclaim]{Question}

\newtheorem{defn-thm}[proclaim]{Definition--Theorem}  

\theoremstyle{bozont-def}

\theoremstyle{bozont-reverse}    

\theoremstyle{bozont-reverse-sc}
\newtheorem{proclaimr-special}[proclaim]{\specialthmname}
{\def\specialthmname{#1}\begin{proclaimr-special}}%
{\end{proclaimr-special}}
\theoremstyle{bozont-remark-reverse}

\newtheorem{SubHeading-r}[proclaim]{\SubHeadingName}
\newtheorem{sSubHeading-r}[equation]{\sSubHeadingName}
\newtheorem{SubHeadingr}[proclaim]{\SubHeadingName}

\theoremstyle{bozont-def-newnum-reverse}    

\theoremstyle{bozont-def-reverse}

\newtheorem{newnumspecial}[proclaim]{\specialnewnumname}

\numberwithin{equation}{proclaim}

\numberwithin{figure}{section}

\newenvironment{enumerate-p}{
  \begin{enumerate}}
  {\setcounter{equation}{\value{enumi}}\end{enumerate}}
\newenvironment{enumerate-cont}{
  \begin{enumerate}
    {\setcounter{enumi}{\value{equation}}}}
  {\setcounter{equation}{\value{enumi}}
  \end{enumerate}}

\newlength{\swidth}
\makeatother

\DeclareMathAlphabet{\smallchanc}{OT1}{pzc}%
                                 {m}{it}
\DeclareFontFamily{OT1}{pzc}{}
\DeclareFontShape{OT1}{pzc}{m}{it}%
             {<-> s * [1.100] pzcmi7t}{}
\DeclareMathAlphabet{\mathchanc}{OT1}{pzc}%
                                 {m}{it}

\newcommand{\mcH}{\mathchanc{H}}

\newcommand{\mcT}{\mathchanc{T}}

\newcommand{\mcm}{\mathchanc{m}}

\newcommand{\mco}{\mathchanc{o}}

\newcommand{\mcr}{\mathchanc{r}}

\DeclareFontFamily{OMS}{rsfs}{\skewchar\font'60}
\DeclareFontShape{OMS}{rsfs}{m}{n}{<-5>rsfs5 <5-7>rsfs7 <7->rsfs10 }{}
\DeclareSymbolFont{rsfs}{OMS}{rsfs}{m}{n}
\DeclareSymbolFontAlphabet{\scr}{rsfs}

\newcommand{\sC}{\scr{C}}

\newcommand{\sE}{\scr{E}}
\newcommand{\sF}{\scr{F}}
\newcommand{\sG}{\scr{G}}
\newcommand{\sH}{\scr{H}}
\newcommand{\sI}{\scr{I}}

\newcommand{\sO}{\scr{O}}
\newcommand{\sP}{\scr{P}}

\newcommand{\sR}{\scr{R}}

\newcommand{\sX}{\scr{X}}

\newcommand{\bC}{\mathbb{C}}

\newcommand{\bP}{\mathbb{P}}
\newcommand{\bQ}{\mathbb{Q}}

\newcommand{\bZ}{\mathbb{Z}}

\newcommand{\fM}{\mathfrak{M}}

\DeclareMathOperator{\depth}{{depth}}

\newcommand{\sHom}[0]{{\mcH\mco\mcm}}    
    
\newcommand{\sTor}[0]{{\mcT\mco\mcr}}    

\DeclareMathOperator{\id}{{id}}

\DeclareMathOperator{\red}{red}

\DeclareMathOperator{\Spec}{{Spec}}
\DeclareMathOperator{\supp}{{supp}}

\newcommand{\factor}[2]{\left. \raise 2pt\hbox{\ensuremath{#1}} \right/
        \hskip -2pt\raise -2pt\hbox{\ensuremath{#2}}}

\def\coh#1.#2.#3.{H^{#1}(#2,#3)}
\def\dimcoh#1.#2.#3.{h^{#1}(#2,#3)}
\def\hypcoh#1.#2.#3.{\mathbb H_{\vphantom{l}}^{#1}(#2,#3)}
\def\loccoh#1.#2.#3.#4.{H^{#1}_{#2}(#3,#4)}
\def\dimloccoh#1.#2.#3.#4.{h^{#1}_{#2}(#3,#4)}
\def\lochypcoh#1.#2.#3.#4.{\mathbb H^{#1}_{#2}(#3,#4)}
\def\ses#1.#2.#3.{0  \longrightarrow  #1   \longrightarrow 
 #2 \longrightarrow #3 \longrightarrow 0} 
\def\sesshort#1.#2.#3.{0
 \rightarrow #1 \rightarrow #2 \rightarrow #3 \rightarrow 0}
\def\dist#1.#2.#3.{  #1   \longrightarrow 
 #2 \longrightarrow #3 \stackrel{+1}{\longrightarrow} } 
\def\CDdist#1.#2.#3.{  #1   @>>>  #2  @>>>   #3 @>+1>> }  
\def\shortses#1.#2.#3.{0  \rightarrow  #1   \rightarrow 
 #2  \rightarrow   #3 \rightarrow  0}
\def\shortdist#1.#2.#3.{  #1   \rightarrow 
 #2  \rightarrow   #3 \stackrel{+1}{\rightarrow} }  
\def\ddist#1.#2.#3.#4.#5.#6.{\CD
#1 @>>> #2 @>>> #3 @>+1>> \\
@VVV @VVV @VVV \\
#4 @>>> #5 @>>> #6 @>+1>> 
\endCD}
\def\ddistun#1.#2.#3.#4.#5.#6.{\CD
#1 @>>> #2 @>>> #3 @>+1>> \\
@. @VVV @VVV  \\
#4 @>>> #5 @>>> #6 @>+1>> 
\endCD}
\def\Iff#1#2#3{
\hfil\hbox{\hsize =#1
\vtop{\noin #2}
\hskip.5cm 
\lower.5\baselineskip\hbox{$\Leftrightarrow$}\hskip.5cm
\vtop{\noin #3}}\hfil\medskip}
\newcommand{\union}\cup
\newcommand{\intersect}\cap
\newcommand{\Union}\bigcup
\newcommand{\Intersect}\bigcap
\def\myoplus#1.#2.{\underset #1 \to {\overset #2 \to \oplus}}

\widepagestyle

\title[Base change behavior of the relative canonical sheaf]{Base change behavior of the relative canonical sheaf related to higher dimensional moduli}

\author{Zsolt Patakfalvi}

\makeatletter

\setitemize[1]{leftmargin=*,parsep=0em,itemsep=0.125em,topsep=0.125em}
\makeatother
\address{Zsolt Patakfalvi, University of Washington, Department of Mathematics, Box 354350,
Seattle, WA 98195, U.S.A.
}
\email{pzs@math.washington.edu}
\urladdr{http://www.math.washington.edu/\~{}pzs}

\theoremstyle{plain}

\begin{document}

\begin{abstract} 
We show that the compatibility of the relative canonical sheaf with base change fails generally in families of normal varieties. Furthermore, it always fails if the general fiber of a family of pure dimension $n$ is Cohen-Macaulay and the special fiber contains a strictly $S_{n-1}$ point. In particular, in moduli spaces with functorial relative canonical sheaves Cohen-Macaulay schemes can not degenerate to $S_{n-1}$ schemes. Another, less immediate consequence is that the canonical sheaf of an $S_{n-1}$, $G_2$ scheme of pure dimension $n$ is not $S_3$. 
\end{abstract}

\maketitle

\tableofcontents


\section{Introduction}
\label{sec:introduction}

The canonical sheaf plays a crucial role in the classification of varieties of characteristic zero. Global sections of its powers define the canonical map, which is birational onto its image  for varieties of general type with mild singularities. The image is called the canonical model, and it is a unique representative of the birational equivalence class of the original variety. In particular, the canonical model can be used to construct a moduli space that classifies varieties of general type up to birational equivalence. This moduli space $\overline{\fM}_h$ of stable schemes, is the higher dimensional generalization of the intensely investigated space $\overline{\fM}_g$ of stable curves. In order to build $\overline{\fM}_h$,  it is important to understand when  the canonical sheaf behaves functorially in families, that is, when it is  compatible with base change. 

More precisely, to obtain a compact moduli space, in $\overline{\fM}_h$ not only canonical models are allowed, but also their generalizations, the semi-log canonical models \cite[Definition 15]{KJ_MOV}. By definition these are projective schemes with semi-log canonical singularities \cite[Definition 3.13.5]{HCD_KSJ_RA} and ample canonical bundles. The first naive definition of the moduli functor of stable schemes with Hilbert function $h$ is then as follows. Here $h : \bZ \to \bZ$ is an arbitrary function.
\begin{equation}
\label{eq:naive}
\overline{\fM}_h(B) = \factor{\left\{ f : X \to B  \left| \parbox{160pt}{$f$ is flat, proper, $X_{\bar{b}}$ is a semi-log canonical model ($\forall b \in B$), $h(m)= \chi ( \omega_{X_b}^{[m]} ) \  (\forall m \in \bZ, b \in B)$}  \right. \right\}}{\parbox{50pt}{\qquad \\[20pt] $\cong$  over $B$}}
\end{equation}
As usual, the naive definition works only in the naive cases but not in general. More precisely, \eqref{eq:naive} is insufficient to prove the existence of a projective coarse moduli space or a proper Deligne-Mumford stack structure on $\overline{\fM}_h$ (c.f., \cite{KJ_MOV},\cite{KJ_HAH}). In general, \eqref{eq:naive} has to be complemented with:
\begin{equation}
             \label{eq:Kollar}
               \left. \omega_{X/B}^{[m]} \right|_{X_b} \cong \omega_{X_b}^{[m]} \textrm{ for every integer } m \textrm{ and } b \in B .
\end{equation}
Usually \eqref{eq:Kollar} is referred to as \emph{Koll\'ar's condition} (e.g. \cite[page 238]{HB_KSJ_RPB}). Note also that \eqref{eq:Kollar} is not necessary for reduced $B$, but it does add important extra restrictions when $B$ is non-reduced.

Currently, it is not understood in every aspect why and how deeply this condition is needed. For example it is not known if in characteristic zero it is equivalent or not to the other possible choice, called \emph{Viehweg's condition}
(see \cite[Assumption 8.30]{VE_QPM}  or \cite[page 238]{HB_KSJ_RPB}):
\begin{equation}
  \label{eq:line_bundle}
  \textrm{there is an integer $m$ such that }   \omega_{X/B}^{[m]} \textrm{ is a line bundle}.
\end{equation}
The starting point of this article is the $m=1$ case of \eqref{eq:Kollar}, that is,  the compatibility of the relative canonical sheaf with base change. We will try to understand how restrictive this condition is on flat families. The results will also yield statements about how Serre's $S_n$ condition behaves in families and  for the canonical sheaves of single schemes.

Recently it has been proven in \cite[Theorem 7.9.3]{KJ_KS_LCS}, that the relative canonical sheaf of flat families of projective schemes (over $\bC$) with Du Bois fibers is compatible with base change. According to \cite[Theorem 1.4]{KJ_KS_LCS} this pertains to families with semi-log canonical fibers as well.  Furthermore, compatibility holds whenever the fibers are Cohen-Macaulay \cite[Theorem 3.6.1]{CB_GD}  . 

It is important to note at this point, that the $m=1$ case of \eqref{eq:Kollar}, behaves differently than the rest. For $m >1$ there are examples of families of normal surfaces for which \eqref{eq:Kollar} does not hold (c.f., \cite[Section 14.A]{HCD_KSJ_RA}).  However, since normal surfaces are Cohen-Macaulay, condition \eqref{eq:Kollar} with $m=1$ holds for every flat family of normal surfaces. Hence, any incompatibility can be observed only in higher dimensions. Partly due to this fact, there has been  a common misbelief, sometimes even stated in articles, that the relative canonical sheaf is compatible with base change for flat families of normal varieties. The question if this compatibility holds indeed was asked about the same time independently by J\'anos Koll\'ar and the author.

\begin{question} {}
[Koll\'ar]
Is $\omega_{X/B}|_{X_b} \cong \omega_{X_b}$ for every flat family $ X \to B$ of normal varieties?
\end{question}

Here we construct examples showing that the answer is no.  That is, there are flat families of normal varieties over smooth curves such that the relative canonical sheaves are not compatible with base change. The examples also show that the known results are optimal in many senses. That is, the fibers of the given families can be chosen to be $S_{j}$ for any $n > j \geq 2$ and their relative canonical sheaves to be $\bQ$-line bundles. The precise statement is as follows. 

\begin{thm} [(= Corollary \ref{cor:result1})]
\label{thm:main1}
For each $n \geq 3$ and $n>j \geq 2$ there is a flat family $\sH \to B$ of  $S_{j}$ (but not $S_{j+1}$), normal varieties of dimension $n$ over some open set $B \subseteq \bP^1$, with $\omega_{\sH/B}$ a $\bQ$-line bundle, such that
\begin{equation}
\label{eq:main_theorem}
\left. \omega_{\sH/B} \right|_{\sH_0} \not\cong \omega_{\sH_0} ,
\end{equation}
(Here $\sH_0$ is the central fiber of $\sH$.)

Moreover, the general fiber of $\sH$ can be chosen to be smooth and the central fiber to have only one singular point.
\end{thm}

When $j=n-1$ and the general fiber is Cohen-Macaulay, somewhat surprisingly,  the incompatibility of \eqref{eq:main_theorem} always holds. Furthermore, one can allow $S_{n-1}$ points also in the general fibers provided the relative $S_{n-1}$ locus has a components in the central fiber. The precise statement is as follows. (See Section \ref{sec:notation} for the assumptions of the article, e.g., scheme is always separated and of finite type over $k= \bar{k}$, etc.)

\begin{thm} [ (= Theorem \ref{thm:central_fiber_S_n-1} ) ]
\label{thm:main2}
If $f : \sH \to B$ is a flat family of schemes of pure dimension $n$ over a smooth  curve, such that a component of the locus
\begin{equation*}
 \overline{\{ x \in \sH | x \textrm{ is closed, } \depth \sO_{\sH_{f(x)},x} = n-1 \} }
\end{equation*}
is contained in the special fiber $\sH_0$, then  the restriction homomorphism $\omega_{\sH/B}|_{\sH_0} \to \omega_{\sH_0}$ is not an isomorphism.
\end{thm}

In particular, the contrapositive of  Theorem \ref{thm:main2} when the general fiber is Cohen-Macaulay yields the following corollary. 

\begin{cor} [ (= Corollary \ref{cor:no_S_n-1}) ]
\label{cor:main}
If $f : \sH \to B$ is  a flat family   of  schemes of pure dimension $n$ such that $\omega_{\sH/B}$ is compatible with base change  and the general fiber of $f$ is Cohen-Macaulay, then the central fiber of $f$ cannot have a closed point $x$, such that $\depth \sO_{\sH_{f(x)},x}=n-1$.
\end{cor}

Corollary \ref{cor:main} has many geometric consequences with respect to building moduli spaces with functorial relative canonical sheaves. For example, cone singularities over abelian surfaces can not be smoothed over irreducible bases. It also generalizes some aspects of theorems by Koll\'ar and Kov\'acs \cite[Theorem 7.12]{KJ_KS_LCS} and Hassett \cite[Theorem 1.1]{HB_SL} stating that if all fibers are Du Bois schemes or log canonical surfaces and the general fiber is $S_k$ or Cohen-Macaulay, respectively, then so is the central fiber. 

Interestingly,  the non-existence of a depth $n-1$ point is the strongest implication of the compatibility of the relative canonical sheaf with base change. 

\begin{prop} [(= Proposition \ref{prop:sharp})]
\label{prop:main}
Corollary \ref{cor:main} is sharp in the sense that $n-1$ cannot be replaced by $i$ for any $i<n-1$. 
\end{prop}

Summarizing, Corollary \ref{cor:main} and Proposition \ref{prop:main} state that in moduli spaces satisfying Koll\'ar's condition, $S_{n-1}$ schemes do not appear in the irreducible components containing Cohen-Macaulay schemes. However, $S_j$ schemes can possibly show up for some $j<n-1$.

If a scheme $X$ is Cohen-Macaulay, which by definition means that $\sO_X$ is Cohen-Macaulay, then $\omega_X$ is Cohen-Macaulay as well \cite[Corollary 5.70]{KJ_MS_BG}.   One would expect that if $\sO_X$ is only $S_{n-1}$, then typically $\omega_X$ is also $S_{n-1}$ or at least it can be $S_{n-1}$. Surprisingly the truth is quite the opposite. The following application of Theorem \ref{thm:main2} states that in certain cases an $S_{n-1}$ scheme cannot have even an $S_3$ canonical sheaf. 

\begin{thm}[(=Theorem \ref{thm:structure_canonical_Serre})]
\label{thm:main3}
If $X$ is an  $S_3, G_2$ scheme of pure dimension $n$, which has a closed point with depth $n-1$, then $\omega_X$ is not $S_3$.
\end{thm}

The most immediate consequences of Theorem \ref{thm:main3} deal with compatibility of restriction to subvarieties. For example, one can show that on a cone $X$ over a Calabi-Yau threefold $Y$ with $h^2(\sO_Y) \neq 0$, for an effective, normal  Cartier divisor $D$,
\begin{equation*}
\omega_X(D)|_D \cong \omega_D \Leftrightarrow \textrm{$D$ does not pass through the vertex} .
\end{equation*}
Or more generally, for an $S_{n-1}$, normal variety $X$ and an effective, normal Cartier divisor $D$,
\begin{equation*}
\omega_X(D)|_D \cong \omega_D \Leftrightarrow \textrm{$D$ does not pass through any closed point with depth $n-1$} .
\end{equation*}

Theorem \ref{thm:main3} can also  be related to log canonical centers. If $(X,D)$ is a log canonical pair, $D \sim_{\bQ} -K_X$  and $\omega_X$ is not $S_3$ at $x \in X$, then $x$ is a log canonical center of  the pair $(X, D)$ \cite[Theorem 3]{KJ_ALV}. Hence by Theorem \ref{thm:main3}, if $X$ is  $S_{n-1}$ and  $(X,D)$ log canonical such that $D \sim_{\bQ} -K_X$, then $(X, D)$ has a log-canonical center at all closed points with depth $n-1$. This statement is of course obvious if we know that the depth $n-1$ closed points are already log-canonical centers of $X$. However, that is not always the case. For example, let $X$ be the cone, with high enough polarization, over the product $Y$ of a K3 surface with the projective line and let $D$ be the cone over an anti canonical divisor of $Y$. Then, $(X,D)$ is log-canonical,  $X$ is $S_{n-1}$ and the cone point is the only closed point with depth $n-1$ (c.f., Lemma \ref{lem:S_d_cones}). Still, the vertex is not a log-canonical center of $X$, because $K_X$ is not $\bQ$-Cartier.

Theorem \ref{thm:main3} raises  the following question as well. 

\begin{question}
\label{qtn:S_n-1}
Is it true that if $X$ is a pure $n$-dimensional scheme such that $\sO_X$ is $S_l$, but not $S_{l+1}$, and $\omega_X$ is $S_j$, but not $S_{j+1}$, for some $j,l<n$, then $j+l \leq n+1$?
\end{question}

\begin{rem}
By the methods of Section \ref{sec:S_d_cones}, the answer to Question \ref{qtn:S_n-1} is positive if $X$ is a cone over a smooth projective variety. 
\end{rem}

There are a couple of intuitive reasons for the failure of compatibility in \eqref{eq:main_theorem}. First, compatibility holds for the relative dualizing complex if the base is smooth by Proposition \ref{prop:base_change_complex}.\ref{itm:base_change_complex:restriction_isomorphism}. Hence  $\omega_{\sH/B}$ is a non-functorial component, the $-n$-th cohomology sheaf, of the functorial object $\omega_{\sH/B}^{\bullet}$. For example, by the proof of Theorem \ref{thm:main2}, if the general fiber is Cohen-Macaulay and the central fiber is $S_{n-1}$, the restriction homomorphism fits into an exact sequence as follows, with a non-zero term on the right.
\begin{equation}
\label{eq:tor}
\xymatrix{ 
0 \ar[r] & \omega_{\sH/B}|_{X_b} \ar[r] & \omega_{\sH_b} \ar[r] & \sTor^1(h^{-(n-1)}(\omega_{\sH/B}^{\bullet}), \sO_{\sH_0}) \ar[r] & 0 
} 
\end{equation}
This shows in a precise way, how the functoriality might be destroyed by passing to the lowest cohomology sheaf of $\omega_{\sH/B}^{\bullet}$.

Another explanation for the incompatibility \eqref{eq:main_theorem} is that $\sH_0$ is too singular. Using stable reduction one may find a replacement for $\sH_0$ with the mildest possible singularities. The reduction steps consist of blow-ups, finite surjective normalized base changes and contractions on the total space of the family. The output is a family, the relative canonical sheaf of which is compatible with base change by \cite{KJ_KS_LCS}. At the end of the article, we also present the stable reduction of our construction using a straight forward ad-hoc method. The algorithmic, and lengthy, method can be found in the preprint version of the article.

In Section \ref{sec:background}, we start with a short background overview on the base-change properties of relative dualizing complexes and relative canonical sheaves. The proofs of the main theorems can be found in Section \ref{sec:construction_of_families} and Section \ref{sec:Cohen-Macaulay}. Some of these results are based on the existence of projective cones with appropriately chosen singularities.  In Section \ref{sec:S_d_cones} we give a cohomological characterization of when certain sheaves on a cone are $S_d$. Then in Section \ref{sec:construction_of_varieties} we use this characterization to give the desired examples of projective cones. In Section \ref{sec:stable_reduction} we compute the stable limit of our construction.

\paragraph{\textbf{Acknowledgements.}} The discussion contains ideas that originated from J\'anos Koll\'ar and my advisor S\'andor Kov\'acs. I would like to thank both of them for their help. I would also like to thank Joseph Lipman for useful comments.

\section{Notations and assumptions}
\label{sec:notation}

Unless otherwise stated \emph{scheme} means a separated scheme of finite type over  a fixed field $k$ of characteristic zero and every morphism is separated.  A \emph{variety} is an integral scheme.  A \emph{projective} or \emph{quasi-projective} scheme  means a projective or quasi-projective scheme  over $k$. A \emph{curve} is a quasi-projective, integral scheme of dimension one. If $Y$ is a subscheme of $X$, then $\sI_{Y,X}$ is the ideal sheaf of $Y$ in $X$. If $\sI_{Y,X}$ is a line bundle (i.e. a locally free sheaf of rank one), then we define $\sO_X(-Y):=\sI_{Y,X}$ and $\sO_X(Y):= \sO_X(-Y)^{-1}$. Notice, that $\sI_{Y,X}$ being a line bundle is equivalent, to $Y$ being defined around every point $P$ by a single non zero divisor element of $\sO_{X,P}$.  

A hypersurface of a quasi-projective scheme $X \subseteq \bP^N$ is a subscheme $H \subseteq X$ defined by a section of $\sO_{X}(d)$ for some $d>0$.
If $H$ and $H'$ are hypersurfaces of a quasi-projective scheme $X \subseteq \bP^N$, defined by $f_0$ and $f_{\infty} \in H^0(\bP^N, \sO_{\bP^N}(d))$, respectively, then the \emph{pencil} generated by $H$ and $H'$ is the subscheme $\sH \subseteq X \times \bP^1$ defined by the section $f_0 t_0 + f_{\infty} t_1$ of $H^0(X \times \bP^1, \sO(d,1))$. Here $t_0$ and $t_1$ are the usual parameters of $\bP^1$, and $f_0$ and $f_{\infty}$ are viewed as elements of $H^0(X, \sO_X(d))$ via the natural homomorphism $H^0(\bP^N, \sO_{\bP^N}(d)) \to H^0(X, \sO_X(d))  $.

For a complex $\sC^{\bullet}$ of sheaves, $h^i(\sC^{\bullet})$ is the $i$-th cohomology sheaf of $\sC$. For a morphism $f : X \to Y$, $\omega_{X/Y}^{\bullet}:= f^! \sO_Y$, where $f^!$ is the functor obtained in \cite[Corollary VII.3.4.a]{HR_RAD}. If $f$ has equidimensional fibers of dimension $n$, then $\omega_{X/Y} := h^{-n} (\omega_{X/Y}^{\bullet})$. Every complex and morphism of complexes is considered in the derived category $D(qc/\_)$ of quasi-coherent sheaves up to the equivalences defined there. If $Z$ is a closed subscheme of $X$, where $\iota : Z \to X$ is the embedding morphism, then the map $\sR \iota_* \cong \iota_*$ identifies $D(qc/Z)$ with a full subcategory of $D(qc/X)$. We use this identification at multiple places, equating $\sC^{\bullet}$ and $\sR \iota_* \sC^{\bullet}$ for every $\sC^{\bullet} \in D(qc/Z)$.  If $Z$ is a closed subscheme of a scheme $X$, then $(\_)|^L_Z$ denotes the derived restriction functor, which is naturally isomorphic to $\_ \otimes^L \sO_Z$ via the above mentioned identification.   A \emph{line bundle} is a locally free sheaf of rank one.

If $X \to B$ is a morphism of schemes, then $X_b$ is the scheme theoretic fiber of $X$ over $B$. If a sheaf $\sF$ on $X$ is given, then $\sF_b:= \sF|_{X_b}$. The dimension $\dim_X P$ of a point $P \in X$ is the dimension of its closure in $X$. The acronym \emph{slc} stands for semi-log canonical (\cite[Definition 3.13.5]{HCD_KSJ_RA}). 

The \emph{depth}  of a coherent sheaf $\sF$ at a point $x \in X$ is by definition the depth of $\sF_x$ with respect to the maximal ideal $m_{X,x}$ at $x$ and is denoted by $\depth \sF_x$. The depth of a scheme $X$ at $x$ is $\depth \sO_{X,x}$.  A coherent sheaf $\sF$ is $S_d$ on $X$ if for every $x \in X$,
\begin{equation}
\label{eq:depth_S_d}
\depth \sF_x \geq \min \{d,\dim \sO_{X,x}\}.
\end{equation}
Note, that there is an ambiguity in the literature about the definition of $S_d$ sheaves. Many sources replace $\sO_{X,x}$ in \eqref{eq:depth_S_d} by $\sF_x$, thus gaining a weaker notion. Since, every sheaf of this article has full support, or equivalently every sheaf is considered over its support, the two definitions are equivalent for all cases considered here. Hence, we decided to include the stronger notion, but the reader should feel free to think about the other one as well. For a morphism $f : X \to B$, $\sF$ is \emph{relative} $S_d$ if $\sF|_{X_{b}}$ is $S_d$ for all $b \in B$.  The word (relative) \emph{Cohen-Macaulay} is a synonym for (relative) $S_{\dim X}$.

A scheme $X$ is $G_r$ for some $r \geq 0$ if it is Gorenstein in codimension $r$. A point $P \in X$ is an \emph{associated point} of a coherent sheaf $\sF$ if $m_{X,P}$ is the annihilator of some element of $\sF_P$. An \emph{associated component} of a coherent sheaf is the closure of an associated point. One can show that if $Q \in X$, $\sF_Q \neq 0$ and $\sP$ is the set of prime ideals of $\sO_{X,Q}$ corresponding to generalizations of $Q$ that are associated points of $\sF$, then \begin{equation*}
\bigcup_{P \in \sP} P = \{x \in \sO_{X,Q} | \exists 0 \neq m \in \sF_Q : x m = 0 \}
\end{equation*}
Consequently, if $s$ is a section of a line bundle, then it does not vanish on any associated component of $X$ (i.e., of $\sO_X$) if and only if for every $P \in X$, $s_P$ is not a zero divisor. That is, if $H$ is the subscheme of $X$ cut out by $s$, then $\sI_{H,X}$ is a line bundle if and only if $s$ does not vanish on any associated component of $X$.

 For an $S_2$, $G_1$ scheme and an arbitrary coherent sheaf $\sF$, the $n$-th \emph{reflexive power} is
\begin{equation*}
\sF^{[n]} :=
\left\{
\begin{matrix}
(\sF^{\otimes n})^{**}  & \textrm{ if } n \geq 0  \\
(\sF^{\otimes (-n)})^* & \textrm{ if } n < 0 .
\end{matrix}
\right.
\end{equation*}
That is, it is the reflexive hull of the $n$-th tensor power. A coherent sheaf $\sF$ is a $\bQ$-line bundle, if  $\sF^{[n]}$ is a line bundle for some $n>0$.   Note, that if $f : X \to B$ is a family with $\omega_{X_b}$ a $\bQ$-line bundle for all $b \in B$, then $\omega_{X/B}$ is not necessarily a $\bQ$-line bundle  \cite[Section 14.A]{HCD_KSJ_RA}. However, if $X_b$ are $S_2,G_1$ schemes and $\omega_{X/B}$ a $\bQ$-line bundle then $\omega_{X_b}$ is a $\bQ$-line bundle for all $b \in B$ (c. f. \cite[Lemma 2.6]{HB_KSJ_RPB}). 

\section{Background on base change for dualizing complexes}
\label{sec:background}

This section contains a general overview on the base change properties of  relative dualizing complexes and  relative canonical sheaves. For experts, some of the statements might be well known, still they are included here for completeness and easier reference. Readers more interested in  geometric arguments and willing to accept the statements of this section without  proofs should feel free to skip to the next section.

Recall that the relative dualizing complex $\omega_{X/B}^{\bullet}$ of a quasi-projective family $f : X \to B$ is defined as $f^! \sO_B$. Here $f^!$ is the functor constructed in \cite[Corollary VII.3.4.a]{HR_RAD}. The following technical point should be noted here.

\begin{rem}
There is also another definition of $f^!$ in \cite{NA_TGD} as the right adjoint of $R f_*$. The two definitions coincide for proper morphisms by \cite[Theorem VII.3.3]{HR_RAD} and \cite[Section 6]{NA_TGD}. However, not in general. For example, if $X$ is smooth affine variety over $B = \Spec k$ and $f$ is the structure map, then Hartshorne's definition of $f^! \sO_{\Spec k}$ lives in cohomological degree $- \dim X$ while Neeman's in cohomological degree zero. See \cite[Part I, Exercise 4.2.3.d]{LJ_HM_FOG} for more details on the differences (Neeman's $f^!$ is denoted $f^{\times}$ there).  We use Hartshorne's definition in the present article. 
\end{rem}

The dualizing complex of a single scheme $Y$ is $\omega_Y^{\bullet} :=\omega_{Y/\Spec k}^{\bullet}$. The following fact is  needed in the proof of  Proposition \ref{prop:base_change_complex}.\ref{itm:base_change_complex:S2}. It follows from the invariance of the length of  maximal regular sequences (\cite[Theorem 1.2.5]{BW_HJ_CMR}).

\begin{fact}
\label{fact:S_d_hyperplane}
Let $P$ be a point of a subscheme $H$ of a scheme $X$ such that $(\sI_{H,X})_P$ is a line bundle, $d$ an integer, and  $\sF$ a coherent $S_1$ sheaf with full support (i.e., $\supp \sF = X$) on $X$. Then 
\begin{enumerate}
 \item  \label{itm:S_d_hyperplane:first} $\depth \sF_P \geq d \ \Leftrightarrow \   \depth (\sF|_H)_P \geq d-1$ (here $\sF|_H$ is regarded as a sheaf on $H$, not on $X$),
 \item  \label{itm:S_d_hyperplane:second} $\depth \sF_P \geq \min \{  d, \dim \sO_{X,P} \} \ \Leftrightarrow \  \depth (\sF|_H)_P \geq \min \{ d-1, \dim \sO_{H,P} \}$.
\end{enumerate}

\end{fact}

\begin{prop}
\label{prop:base_change_complex}
Given a flat family $f: \sH \to B$ of schemes of pure dimension $n$ over a smooth base, a point $0 \in B$ and a single quasi-projective scheme $X$ of pure dimension $n$,
\begin{enumerate}
\item \label{itm:base_change_complex:restriction_isomorphism}
there is an isomorphism 
\begin{equation}
\label{eq:base_change_complex:restriction_isomorphism}
\omega_{\sH/B}^{\bullet}|_{\sH_0}^L \cong \omega_{\sH_0}^{\bullet} ,
\end{equation}
\item \label{itm:base_change_complex:restriction_homomorphism}   fixing any  isomorphism in \eqref{eq:base_change_complex:restriction_isomorphism}, yields  natural homomorphism
\begin{equation}
\label{eq:base_change_complex:restriction_homomorphism}
\omega_{\sH/B}|_{\sH_0} \to \omega_{\sH_0} ,
\end{equation}

\item \label{itm:base_change_complex:relative_over_smooth} if $B$ is  of pure dimension $d$ with $\sO_B \cong \omega_B$, then $\omega_{\sH/B}^{\bullet} \cong \omega_{\sH}^{\bullet}[-d]$,
\item \label{itm:base_change_complex:open_set} if $V \subseteq X$ is any open set, then  $\omega_{V}^{\bullet} \cong \omega_X^{\bullet}|_V$,
\item \label{itm:base_change_complex:relative_open_set} if $U \subseteq \sH$ is any open set, then  $\omega_{U/B}^{\bullet} \cong \omega_{\sH/B}^{\bullet}|_U$,
\item \label{itm:base_change_complex:Sd_at_closed_point} if $P \in X$ is a point, then $\depth_P \sO_X = d$ if and only if $h^i(\omega_{X}^{\bullet})_P$ is zero for $i > - d - \dim_X P$ and non-zero for $i=-d - \dim_X P$,
\item \label{itm:base_change_complex:relative_Sd_at_closed_point} if $P \in \sH$ is a point, then $\depth_P \sO_{\sH_{f(P)}}=d$   if and only if $h^i(\omega_{\sH/B}^{\bullet})_P$ is zero for $i > - d - \dim_{\sH_{f(P)}} P$ and non-zero for $i=-d - \dim_{\sH_{f(P)}} P$,
\item \label{itm:base_change_complex:omega_S2} $\omega_X$ is $S_2$,
\item \label{itm:base_change_complex:relative_omega_S2} $\omega_{\sH/B}$ is $S_2$,
\item \label{itm:base_change_complex:Cohen_Macaulay} if the fibers of $f$ are Cohen-Macaulay  then $\omega_{\sH/B} \cong \omega_{\sH/B}^{\bullet}$ and consequently \eqref{eq:base_change_complex:restriction_homomorphism} is an isomorphism,
\item \label{itm:base_change_complex:S2} if $\sH_0$ is $S_2$ and $G_1$, then \eqref{eq:base_change_complex:restriction_homomorphism} is isomorphism if and only if  
\begin{equation}     
\label{eq:base_change_complex:S3_over_H0}                                                                                                                                                               \depth \omega_{\sH/B,P} \geq \min \{  3, \dim \sO_{\sH,P} \} \textrm{, for every }  P \in \sH_0.
\end{equation}
 Furthermore if \eqref{eq:base_change_complex:S3_over_H0} is not satisfied then not only \eqref{eq:base_change_complex:restriction_homomorphism} is not an isomorphism, but $\omega_{\sH/B}|_{\sH_0} \not\cong \omega_{\sH_0}$.
\end{enumerate}
\end{prop}

\begin{proof}

First, we prove point \eqref{itm:base_change_complex:restriction_isomorphism}. It will be an ad-hoc proof, since we have not found the exact statement  in the literature. The statements we found are either only for flat base change morphisms \cite[Corollary VII.3.4.a]{HR_RAD} or for proper $f$ \cite[Part I, Corollary 4.4.3]{LJ_HM_FOG}. Note, that however, it might seem that point \eqref{itm:base_change_complex:restriction_isomorphism} follows from base change for proper $f$, to the best knowledge of the author, it is not clear whether one can compactify a flat morphism to a flat morphism. 

First,  by \cite[Corollary VII.3.4.a]{HR_RAD}, $\omega_{\sH/B}^{\bullet}$ is compatible with flat base change. So, since $\Spec \widehat{\sO_{B,0}}$ is flat over $B$, we may assume that $B$ is the spectrum of a complete local ring of a smooth scheme and $0$ is the unique closed point. In particular, then $B \cong \Spec k[[x_1,\dots,x_{m}]]$. Hence, by induction on $m$, it is enough to prove that 
\begin{equation}
\label{eq:to_prove}
\omega_{\sH/B}^{\bullet}|_Y \cong \omega_{Y/C}^{\bullet} ,
\end{equation}
where  $C := \Spec k[[x_1,\dots,x_{m-1}]]$ and $Y:= \sH \times_B C$. To prove \eqref{eq:to_prove} first consider the usual exact triangle
\begin{equation}
\label{eq:structure_sheaves}
\xymatrix{
\sO_{\sH} \ar[r]^{\mu} & \sO_{\sH} \ar[r] & \sO_{Y} \ar[r]^{+1} & ,
}
\end{equation}
where $\mu$ is multiplication by $x_m$. Tensoring \eqref{eq:structure_sheaves} by $\omega_{\sH/B}^{\bullet}$ yields
\begin{equation}
\label{eq:tensored}
\xymatrix{
 \omega_{\sH/B}^{\bullet} \ar[rr]^{\mu \otimes \id_{\omega_{\sH/B}^{\bullet}}} & &  \omega_{\sH/B}^{\bullet} \ar[r] & \omega_{\sH/B}^{\bullet}|_{Y}^L \ar[r]^-{+1} & .
}
\end{equation}
On the other hand, applying $\sR \sHom ( \underline{\quad} , \omega_{\sH/B}^{\bullet})$ and a rotation to \eqref{eq:structure_sheaves} yields
\begin{equation}
\label{eq:rotated}
\xymatrix{
 \omega_{\sH/B}^{\bullet} \ar[rr]^{\mu \otimes \id_{\omega_{\sH/B}^{\bullet}}} & & \omega_{\sH/B}^{\bullet} \ar[r]  &  \sR \sHom (\sO_Y,\omega_{\sH/B}^{\bullet} )[1]    \ar[r]^-{+1}   &  .
}
\end{equation}
So, \eqref{eq:tensored} and \eqref{eq:rotated} together imply that
\begin{equation}
\label{eq:res}
 \sR \sHom (\sO_Y,\omega_{\sH/B}^{\bullet} )[1] \cong \omega_{\sH/B}^{\bullet}|_{Y} .
\end{equation}
Denote by $\iota$ and $g$ the maps $Y \to X$ and  $Y \to C$, respectively. The following stream of isomorphisms finishes then the proof of point \eqref{itm:base_change_complex:restriction_isomorphism}.
\begin{multline*}
\label{eq:rel_can}
\omega_{\sH/B}^{\bullet}|_Y 
\cong \underbrace{ \sR \sHom_X( \sR \iota_* \sO_Y,  \omega_{\sH/B}^{\bullet} ) [1]}_{\textrm{by \eqref{eq:res}}}
\cong \underbrace{ \sR \sHom_Y( \sO_Y,  \iota^! \omega_{\sH/B}^{\bullet} ) [1]}_{\textrm{by Grothendieck duality}} 
\\
\cong \iota^! \omega_{\sH/B}^{\bullet} [1] \cong 
\underbrace{\iota^! f^! \sO_B [1]}_{\textrm{definition of $\omega_{\sH/B}^{\bullet}$}} \cong 
\underbrace{i^! f^! \omega_B^{\bullet} [-(m-1)]}_{\omega_B^{\bullet}[-m] \cong \sO_B}
\\ \cong \underbrace{\omega_Y^{\bullet} [-(m-1)]}_{\omega_Y^{\bullet} \cong (f \circ \iota)^! \omega_B^{\bullet}}
\cong \underbrace{g^! \omega_C^{\bullet} [-(m-1)]}_{\omega_Y^{\bullet} \cong g^! \omega_C^{\bullet}}
\underbrace{\cong g^! \sO_C }_{\omega_C^{\bullet}[-(m-1)] \cong \sO_C} 
\cong \omega_{Y/C}^{\bullet}
.
\end{multline*}

To prove point \eqref{itm:base_change_complex:restriction_homomorphism}, notice that since $\omega_{\sH/B} := h^{-n}( \omega_{\sH/B}^{\bullet})$ is the lowest cohomology sheaf of $\omega_{\sH/B}^{\bullet}$, there is a homomorphism
\begin{equation}
\label{eq:base_change_complex:lowest_coho}
\omega_{\sH/B}[n] \to   \omega_{\sH/B}^{\bullet} .
\end{equation}
Applying $( \_ )|^L_{\sH_0}$ to \eqref{eq:base_change_complex:lowest_coho} and then composing with the isomorphism given by \eqref{eq:base_change_complex:restriction_isomorphism} yields a homomorphism
\begin{equation}
\label{eq:base_change_complex:composition}
\omega_{\sH/B}[n]|^L_{\sH_0} \to \omega_{\sH_0}^{\bullet}.
\end{equation}
Finally taking $-n$-th cohomology sheaves of \eqref{eq:base_change_complex:composition} yields the restriction homomorphism of \eqref{eq:base_change_complex:restriction_homomorphism}.

Point \eqref{itm:base_change_complex:relative_over_smooth} is shown by the following line of isomorphisms.
\begin{equation*}
\omega_{\sH/B}^{\bullet} = f^! \sO_B \cong f^! \omega_B \cong f^! \omega_B^{\bullet}[-d] \cong \omega_{\sH}^{\bullet}[-d]
\end{equation*}

To prove point \eqref{itm:base_change_complex:open_set}, consider the following commutative diagram.
\begin{equation*}
\xymatrix{
V \ar[r]^{j} \ar[dr]^{\nu} & X \ar[d]^{\mu} \\
& \Spec k
}
\end{equation*}
Since $j$ is smooth of relative dimension $0$, using notations of \cite{HR_RAD}, $j^! \cong j^{\#} \cong j^*$ and then
\begin{equation*}
\omega_V^{\bullet} = \nu^! \sO_{\Spec k} \cong j^! \mu^! \sO_{\Spec k} \cong j^! \omega_X^{\bullet} \cong j^* \omega_X^{\bullet} = \omega_X^{\bullet}|_V.
\end{equation*}
Point \eqref{itm:base_change_complex:relative_open_set} follows from points \eqref{itm:base_change_complex:open_set} and \eqref{itm:base_change_complex:relative_over_smooth}.

Point \eqref{itm:base_change_complex:Sd_at_closed_point} is proved in \cite[Proposition 3.2]{KSJ_IC} (by taking $\sF:=\sO_X$). 
To prove point \eqref{itm:base_change_complex:relative_Sd_at_closed_point}, let $b := f(P)$ and consider the following Cartesian square.
\begin{equation*}
\xymatrix{
\sH \ar[d]^f & \sH' \ar[l]^{\lambda'} \ar[d]^{f'} \\
B & \Spec \sO_{B,b} \ar[l]^{\lambda}
}
\end{equation*}
By flat base change,
\begin{equation}
\label{eq:base_change_complex:flat_base_change}
(\lambda')^* \omega_{\sH/B}^{\bullet} \cong \omega_{\sH'/\Spec \sO_{B,b}}^{\bullet} .
\end{equation}
That is, 
\begin{multline*}
h^i \left(\omega_{\sH/B}^{\bullet} \right)_P \cong 
\underbrace{h^i \left(\omega_{\sH'/ \Spec \sO_{B,b}}^{\bullet} \right)_P}_{\textrm{by \eqref{eq:base_change_complex:flat_base_change}}}
\cong \underbrace{h^i \left(\omega_{\sH'}^{\bullet}[- \dim \sO_{B,b}] \right)_P}_{\textrm{by point \eqref{itm:base_change_complex:relative_over_smooth}}} \\ \cong h^{i- \dim \sO_{B,b} } \left( \omega_{\sH'}^{\bullet} \right)_P .
\end{multline*}
Hence 
\begin{equation*}
h^i(\omega_{\sH/B}^{\bullet})_P \textrm{ is } 
\left\{ 
\begin{matrix}
0  & \textrm{, if } i> -d - \dim_{\sH_b} P \\
\neq 0 & \textrm{, if } i = -d - \dim_{\sH_b} P 
\end{matrix}
\right.
\end{equation*}
\begin{equation*}
\Updownarrow
\end{equation*}
\begin{equation*}
h^i(\omega_{\sH'}^{\bullet})_P \textrm{ is } 
\left\{ 
\begin{matrix}
0  & \textrm{, if } i> -d - \dim_{\sH_b} P - \dim \sO_{B,b} \\
\neq 0 & \textrm{, if } i = -d - \dim_{\sH_b} P   - \dim \sO_{B,b}
\end{matrix}
\right.
\end{equation*}
\begin{equation*}
\Updownarrow
\end{equation*}
\begin{equation*}
 \depth_P \sO_{\sH'} = d + \dim \sO_{B,b}
\end{equation*}
\begin{equation*}
\underbrace{\Updownarrow }_{ \textrm{ By Fact \ref{fact:S_d_hyperplane} }  }
\end{equation*}
\begin{equation*}
 \left( \depth_P \sO_{(\sH')_{f(P)}} =  \right) \depth_P \sO_{\sH_{f(P)}} = d 
\end{equation*}

To prove point \eqref{itm:base_change_complex:omega_S2}, by point \eqref{itm:base_change_complex:open_set} we may assume that $X$ is affine. Using point \eqref{itm:base_change_complex:open_set} again we may also assume that it is projective. Then  \cite[Corollary 5.69]{KJ_MS_BG} concludes the proof of point \eqref{itm:base_change_complex:omega_S2}. Point \eqref{itm:base_change_complex:relative_omega_S2} is a consequence of point \eqref{itm:base_change_complex:omega_S2} and point \eqref{itm:base_change_complex:relative_over_smooth}. Point \eqref{itm:base_change_complex:Cohen_Macaulay} is shown in \cite[Theorem 3.5.1]{CB_GD}.

To prove point \eqref{itm:base_change_complex:S2}, notice that  by point \eqref{itm:base_change_complex:omega_S2}, $\omega_{\sH_0}$ is $S_2$. Also since $\sH_0$ is $G_1$, using point \eqref{itm:base_change_complex:Cohen_Macaulay}, the homomorphism $\omega_{\sH/B}|_{\sH_0} \to \omega_{\sH_0}$ is isomorphism in codimension one. Then by  \cite[Theorem 1.9 and Theorem 1.12]{HR_GD}, using that $\sH_0$ is $S_2$ and $G_1$, $\omega_{\sH/B}|_{\sH_0} \to \omega_{\sH_0}$ is isomorphism if and only if $\omega_{\sH/B}|_{\sH_0}$ is $S_2$. Finally, by Fact \ref{fact:S_d_hyperplane}.\ref{itm:S_d_hyperplane:second}, this is equivalent to \eqref{eq:base_change_complex:S3_over_H0}.

Notice that if \eqref{eq:base_change_complex:S3_over_H0} is not satisfied, then $\omega_{\sH/B}|_{\sH_0}$ is not $S_2$ over $\sH_0$. Hence in this case not only \eqref{eq:base_change_complex:restriction_homomorphism} can not be isomorphism, but any isomorphism between $\omega_{\sH/B}|_{\sH_0}$ and $\omega_{\sH_0}$ is impossible.
\end{proof}

\begin{rem}
A priori,  saying that \eqref{eq:base_change_complex:restriction_homomorphism} is an isomorphism is a stronger statement than that $\omega_{\sH/B}|_{\sH_0}$ is isomorphic to $\omega_{\sH_0}$. However, if $B$ is smooth, $\sH_0$ is projective, $S_2$ and $G_1$, they are equivalent by the following argument. In this case $\omega_{\sH_0}$ is $S_2$ and is a line bundle over the Gorenstein locus $U$. Assume that $\omega_{\sH/B}|_{\sH_0} \cong \omega_{\sH_0}$ via an arbitrary isomorphism $\alpha$. Then  $\omega_{\sH/B}|_{\sH_0}$ is also $S_2$ and a line bundle over $U$. Since both are $S_2$, homomorphisms $\omega_{\sH/B}|_{\sH_0} \to \omega_{\sH_0}$ are determined in codimension one, e.g., over $U$. Furthermore, any two isomorphism over $U$ between any two line bundles  differ by multiplication with an element of $H^0(U, \sO_{\sH_0})$, where  $H^0(U, \sO_{\sH_0}) \cong k^*$, by $\sH_0$ being $S_2$ and projective. Since the restriction of the natural morphism $\beta : \omega_{\sH/B}|_{\sH_0} \to \omega_{\sH_0}$ over $U$ is an isomorphism, $\alpha$ differs from $\beta$ over $U$ by a multiplication with an element of $k^*$. However, then the same is true over entire $X$, by the codimension one determination. Hence $\beta$ is also an isomorphism.
\end{rem}

Finally, we conclude with a statement about restriction behavior of relative dualizing complexes and relative canonical sheaves to hypersurfaces. For that we also need a lemma about flatness of hypersurfaces.

\begin{lem}
\label{lem:ideal_sheaves_line_bundles_flat}
If $f :\sX \to B$ is a flat morphism onto a smooth curve and $\sH \subseteq \sX$ is a subscheme for which $\sI_{\sH,\sX}$ is a line bundle, then the following are equivalent
\begin{enumerate}
\item \label{itm:ideal_sheaves_line_bundles_flat:fiber_ideals} $\sI_{\sH_b,\sX_b}$ is a line bundle for every $b \in B$, and
\item \label{itm:ideal_sheaves_line_bundles_flat:flat} $\sH$ is flat over $B$.
\end{enumerate}
In particular, if $f : \sX \to B$ is flat with fibers of pure dimension $n$ and $\sH \subseteq \sX$  is also flat  with $\sI_{\sH,\sX}$ a line bundle, then  fibers of $\sH$ are of pure dimension $n-1$.
\end{lem}

\begin{proof} 
We prove only the equivalence statement, since the addendum follows from $\sI_{\sH_b,\sX_b}$ being line bundles. 

The statement is local on $\sH$. So, fix $P \in \sH$ and let $Q:=f(P)$. By \cite[Proposition 9.1A.a]{HR_AG}, $\sH$ (resp. $\sX$) is flat over $B$  at $P$ if and only if the homomorphisms $\sO_{\sH,P} \to \sO_{\sH,P}$ (resp. $\sO_{\sX,P} \to \sO_{\sX,P}$) induced by multiplication with some power of the local parameter $t$ of $\sO_{B,Q}$ is injective. Furthermore, by induction this is equivalent to the injectivity of  multiplication with the first power $t$. 

The assumptions of the lemma state that $(\sI_{\sH,\sX})_P \subseteq \sO_{\sX,P}$ is generated by a non-zero divisor element $s$. Hence there is a commutative diagram with exact rows and columns as follows.
\begin{equation*}
\xymatrix{
& & 0 & 0 \\
0 \ar[r] & \ker ( \cdot s )  \ar[r]& \sO_{\sX_Q,P} \ar[r]^{\cdot s} \ar[u] & \sO_{\sX_Q,P} \ar[u] \\
& 0 \ar[r] & \sO_{\sX,P} \ar[r]^{\cdot s} \ar[u] & \sO_{\sX,P} \ar[r]^{\cdot s} \ar[u] & \sO_{\sH,P} \ar[r] & 0 \\
& 0 \ar[r] & \sO_{\sX,P} \ar[r]^{\cdot s} \ar[u]^{\cdot t} & \sO_{\sX,P} \ar[r]^{\cdot s} \ar[u]^{\cdot t} & \sO_{\sH,P} \ar[r] \ar[u]^{\cdot t} & 0 \\
& & 0 \ar[u] & 0 \ar[u] & \ker ( \cdot t) \ar[u] \\
& & & & 0 \ar[u]
}
\end{equation*}
By snake lemma applied vertically, $\ker ( \cdot t) = \ker(\cdot s)$. In particular, $\ker ( \cdot t) =0$ if and only if $\ker ( \cdot s)=0$. The former is equivalent to flatness of $\sH \to B$ at $P$ while the latter is equivalent to $\sI_{\sH_Q,\sX_Q}$ being a line bundle at $P$.
\end{proof}

\begin{prop}
\label{prop:restriction_Cartier_divisor}
If $\sX \to B$ is a flat family of pure $n$-dimensional schemes, and $\sH \subseteq \sX$ a flat   subscheme such that $\sI_{\sH,\sX}$ is a line bundle, then
\begin{enumerate}
\item \label{itm:restriction_Cartier_divisor:complex} there is an isomorphism
\begin{equation}
\label{eq:restriction_Cartier_divisor:complex}
\omega_{\sX/B}^{\bullet}(\sH)|^L_{\sH} [-1] \cong \omega_{\sH/B}^{\bullet} ,
\end{equation}
\item \label{itm:restriction_Cartier_divisor:sheaf} there is a homomorphism
\begin{equation}
\label{eq:restriction_Cartier_divisor:sheaf}
\omega_{\sX/B}(\sH)|_{\sH} \to \omega_{\sH/B} ,
\end{equation}
which is isomorphism over the relative Cohen-Macaulay locus of $\sH \to B$.
\end{enumerate}
\end{prop}

\begin{proof}
Notice first that by Lemma \ref{lem:ideal_sheaves_line_bundles_flat}, $\sH$ has equidimensional fibers and hence $\omega_{\sH/B}$ is defined indeed. To prove point \eqref{itm:restriction_Cartier_divisor:complex}, consider the exact sequence
\begin{equation}
\label{eq:hypersurface}
\xymatrix{
0 \ar[r] & \sO_{\sX}  \ar[r] & \sO_{\sX}(\sH) \ar[r] & \sO_{\sH}(\sH) \ar[r] & 0 .
}
\end{equation}
Applying $( \_ )\otimes^L \omega_{\sX/B}^{\bullet}$ to \eqref{eq:hypersurface}, and then translating yields the exact triangle 
\begin{equation}
\label{eq:after_tensoring}
\xymatrix{
 \omega_{\sX/B}^{\bullet}(\sH)|^L_{\sH} [-1] \ar[r] & \omega_{\sX/B}^{\bullet}  \ar[r] & \omega_{\sX/B}^{\bullet}(\sH)  \ar[r]^-{+1} & .
} 
\end{equation}
On the other hand if $\iota : \sH \to \sX$ is the embedding morphism, then
\begin{equation*}
\omega_{\sH/B}^{\bullet} \cong \iota^! \omega_{\sX/B}^{\bullet} =  R \sHom_{\sH}(\sO_{\sH}, \iota^! \omega_{\sX/B}^{\bullet} )  \cong 
\underbrace{R \sHom_{\sX}(\sO_{\sH}, \omega_{\sX/B}^{\bullet} )}_{\textrm{by Grothendieck duality}} 
.
\end{equation*}
Now, applying $R \sHom_{\sX}( \_ , \omega_{\sX/B}^{\bullet})$ to the twist of \eqref{eq:hypersurface} by $\sO_{\sX}(-\sH)$ yields the exact triangle
\begin{equation}
\label{eq:ext_applied}
\xymatrix{
 \omega_{\sH/B}^{\bullet} \cong  R \sHom_{\sX}(\sO_{\sH}, \omega_{\sX/B}^{\bullet} ) \ar[r] & \omega_{\sX/B}^{\bullet} \ar[r] & \omega_{\sX/B}^{\bullet}(\sH) \ar[r]^-{+1} &
}
\end{equation}
Putting together \eqref{eq:after_tensoring} and \eqref{eq:ext_applied} finishes the proof of point \eqref{itm:restriction_Cartier_divisor:complex}.

To prove \eqref{itm:restriction_Cartier_divisor:sheaf}, take the natural map $\omega_{\sX/B}[n] \to \omega_{\sX/B}^{\bullet}$, twist it with $\sO_{\sX}(\sH)$ and then restrict  to $\sH$. This yields the commutative diagram

\begin{equation}
\label{eq:restriction_composition}
\xymatrix{
\omega_{\sX/B}[n-1](\sH)|^L_{\sH} \ar[r]  \ar@/^2pc/[rr] & \omega_{\sX/B}^{\bullet}(\sH)[-1]|^L_{\sH} \ar[r]_(0.7){\underbrace{\tiny  \cong}_{\textrm{ \tiny by point \eqref{itm:restriction_Cartier_divisor:complex}} }} & \omega_{\sH/B}^{\bullet}
}
\end{equation}
Applying then $h^{-(n-1)}(\_)$ to the long composition arrow of \eqref{eq:restriction_composition}, yields the homomorphism \eqref{eq:restriction_Cartier_divisor:sheaf}. 

Let $P$ be a point of $\sH$ which is relatively Cohen-Macaulay over $B$, and let $b$ be the image of $P$ in $B$. By the opennes of the relative Cohen-Macaulay locus, there is a  neighborhood $U$ of $P$, where $X \to B$ is relatively Cohen-Macaulay. In particular, then $\omega_{\sX/B}[n-1] \to \omega_{\sX/B}^{\bullet}[-1]$ is isomorphism over $U$ by Proposition \ref{prop:base_change_complex}.\ref{itm:base_change_complex:Cohen_Macaulay} and hence so is the first arrow of \eqref{eq:restriction_composition}. This proves that  \eqref{eq:restriction_Cartier_divisor:sheaf} is an isomorphism in a neighborhood of $P$, which finishes the proof of point \eqref{itm:restriction_Cartier_divisor:sheaf} as well.
\end{proof}

\begin{rem}
The homomorphisms constructed in Propositions \ref{prop:base_change_complex} and  \ref{prop:restriction_Cartier_divisor}, e.g., the isomorphisms \eqref{eq:base_change_complex:restriction_isomorphism} and \eqref{eq:restriction_Cartier_divisor:complex}, are not canonical in any sense. 
\end{rem}

\section{Serre's condition on projective cones}
\label{sec:S_d_cones}

In this section we consider sheaves on projective cones that are isomorphic to pullbacks from the base outside the vertex. Lemma \ref{lem:S_d_cones} gives a cohomological description of when such sheaves are  $S_d$. Before that we also need a short lemma, Lemma \ref{lem:S_d_pullback}, about how the property $S_d$ pulls  back in flat relatively Cohen-Macaulay families. 

We cite the following fact separately here, because it is used at many places throughout the article, including the aforementioned Lemma \ref{lem:S_d_pullback}.

\begin{fact}{\cite[Proposition 6.3.1]{GA_EGA_IV_II}}
\label{fact:S_d_pullback}
Let $X$ and $Y$ be two noetherian schemes, $f: X \to Y$ a flat morphism, $P \in X$ arbitrary and $\sF$ a coherent $Y$ module. In this situation,
\begin{equation*}
 \depth_{\sO_{X,P}} (f^* \sF)_P =  \depth_{\sO_{Y,f(P)}}\sF_{f(P)} + \depth_{\sO_{X_{f(P)},P}}  \sO_{X_{f(P)},P} .
\end{equation*}
\end{fact}

\begin{lem}
\label{lem:S_d_pullback}
If $\sG$ is a full dimensional coherent $S_d$ sheaf on the scheme $X$, and $f : \sX \to X$ is a flat, relatively Cohen-Macaulay family, then $\sF:= f^* \sG$ is $S_d$ as well.  
\end{lem}

\begin{proof}
For every $ x \in X$,
\begin{align*}
\depth \sF_x 
& =
\underbrace{\depth \sG_{f(x)} + \depth \sO_{\sX_{f(x)},x}}_{\textrm{Fact \ref{fact:S_d_pullback}}}
\\ & =
\underbrace{\depth \sG_{f(x)} + \dim \sO_{\sX_{f(x)},x}}_{\textrm{$\sX_{f(x)}$ is Cohen-Macaulay}} 
\\ &  \geq 
\underbrace{\min\{d, \dim \sO_{X,f(x)}\} + \dim \sO_{\sX_{f(x)},x}}_{\textrm{$\sG$ is $S_d$}} 
\\ &  \geq 
\min\{d, \dim \sO_{X,f(x)} + \dim \sO_{\sX_{f(x)},x}\} 
\\ &  = 
\underbrace{\min\{d, \dim \sO_{\sX,x} \}}_{\dim \sO_{X,f(x)} + \dim \sO_{\sX_{f(x)},x}= \dim \sO_{\sX,x}  \textrm{ by \cite[Theorem 15.1.ii]{MH_CRT}} } .
\end{align*}

\end{proof}

\begin{lem}
\label{lem:S_d_cones}
Assume that we are in the following situation:
\begin{itemize}
\item $Y$ is a projective scheme, 
\item $X$ is the projectivized cone over $Y$,
\item $P$ is the vertex of $X$ and $V:= X \setminus P$,
\item $d$ is an integer, such that $2 \leq d \leq \dim X$ and
\item $\sF$ is a coherent sheaf on $X$, such that $\sF|_V = \pi^* \sG$ for some $S_{d}$ coherent sheaf $\sG$ on $Y$, where $\pi : V \to Y$ is the natural projection.
\end{itemize}
Then the following conditions are equivalent:
\begin{enumerate}
\item $\depth \sF_P \geq d$
\item $\depth \sF_P \geq \min\{d, \dim \sO_{X,P} \}$
\item $\sF$ is $S_d$
\item \label{itm:third} $\sF$ is $S_2$  and $H^i(Y,\sG(n))=0$ for all $0 < i < d -1$ and $n \in \bZ$.
\end{enumerate}
\end{lem}

\begin{proof}
Since $\sG$ is $S_{d}$, $\sF$ is $S_d$ everywhere except at the vertex $P$ by Lemma \ref{lem:S_d_pullback}. Hence, using the assumption $d \leq \dim X$,
\begin{equation*}
\begin{array}{c}
\textrm{$\sF$ is $S_d$,} \\
\Updownarrow \\
\depth \sF_P \geq \min\{d, \dim \sO_{X,P} \}  \\
\Updownarrow \\
\depth \sF_P \geq d \\
\Updownarrow \\
\textrm{$H_P^i(Z,\sF)=0$ for all $i<d$ and for the affine cone $Z$,}
\end{array}
\end{equation*}
where the latter equivalence follows from \cite[Exercise III.3.4.b and Exercise III.2.5]{HR_AG}. So, we are left to show that the condition $H_P^i(Z,\sF)=0$ for all $i<d$ is equivalent to point \eqref{itm:third}. Define $U:= Z \setminus P$. Then there is a long exact sequence:
\begin{equation*}
\xymatrix{
\dots \ar[r] & H_P^i(Z,\sF) \ar[r] & H^i(Z,\sF) \ar[r] & H^i(U,\sF) \ar[r] & \dots 
} .
\end{equation*}
Since $Z$ is affine $H^i(Z,\sF)=0$ for all $i>0$. Hence  
\begin{equation}
\label{eq:isom}
H^i(U,\sF) \cong H_P^{i+1}(Z,\sF) \textrm{  for all } i>0 .
\end{equation}
So, since  $H_P^0(Z,\sF)=H_P^1(Z,\sF)=0$ is assumed  in point \eqref{itm:third}, it is enough to show that for all $0<i<d-1$,
\begin{equation}
\label{eq:claimed}
H^i(U,\sF) \cong \bigoplus_{n \in \bZ} H^i(Y,\sG(n)) .
\end{equation}
In fact we will prove this for all $i$. First, notice that $U \cong \Spec_Y(\bigoplus_{n \in \bZ} \sO_Y(n))$ and the natural projection  $\Spec_Y(\bigoplus_{n \in \bZ} \sO_Y(n)) \to Y$ can be identified with  $\pi|_U$  via this isomorphism. Hence $(\pi|_U)_* \sO_U \cong \bigoplus_{n \in \bZ} \sO_Y(n)$ and $R^i (\pi|_U)_* \sO_U=0$ for $i>0$. So:
\begin{multline*}
H^i(U,\sF) \cong H^i(Y,(\pi|_U)_* \sF|_U) \cong H^i(Y,(\pi|_U)_* (\pi|_U)^* \sG) \cong \\ \cong H^i(Y, \bigoplus_{n \in \bZ} \sG(n)) \cong \bigoplus_{n \in \bZ} H^i(Y,  \sG(n)) 
\end{multline*}
as claimed in \eqref{eq:claimed}.
\end{proof}

\section{Construction of varieties with prescribed singularities}
\label{sec:construction_of_varieties}

In this section, normal $S_{j}$ (but not $S_{j+1}$) varieties of dimension $n \geq 3$  with  $S_l$ (but not $S_{l+1}$), $\bQ$-line bundle canonical sheaves are constructed for certain values of $j$ and $l$. They are going to be used in Section \ref{sec:construction_of_families} and in Section \ref{sec:Cohen-Macaulay} to build families with prescribed base change behavior for the relative canonical sheaves. First we need some lemmas.
 
\begin{lem}
\label{lem:hypersurface}
If  $H$ is a general, high enough degree hypersurface in a projective variety $X$, then $H^i(H,\sO_H) \cong H^i(X,\sO_X)$ for every $0 < i < \dim H$.
\end{lem}

\begin{proof}
We start with the usual exact sequence 
\begin{equation}
\label{eq:egzact}
\xymatrix{
0 \ar[r] & \sO_{X}(-H) \ar[r] & \sO_{X} \ar[r] & \sO_{H} \ar[r] & 0
} . 
\end{equation}
Since $\deg H \gg 0$,
\begin{equation}
\label{eq:Serre_vanishing}
 H^i(X,\sO_{X}(-H))=0 \textrm{ whenever } i< \dim X .
\end{equation}
Taking the cohomology long exact sequence of \eqref{eq:egzact} and using \eqref{eq:Serre_vanishing} finishes the proof.
\end{proof}

Iterated use of Lemma \ref{lem:hypersurface} yields the following 

\begin{lem}
\label{lem:complete_intersection}
If  $H$ is a general, high enough degree complete intersection (i.e. it is the intersection of hypersurfaces, all of which are high enough degree)  in a smooth projective variety $X$, then $H^i(H,\sO_H) \cong H^i(X,\sO_X)$ for every $0 < i < \dim H$.
\end{lem}

Finally, iterated use of the adjunction formula yields the following.

\begin{lem}
\label{lem:complete_intersection_adjunction}
If  $H$ is a  complete intersection in a smooth projective variety $X$, then $\omega_H \cong \omega_X (m)|_H$ for some $m >0$ (here $\sO_X(1)$ is the very ample line bundle given by the projective embedding of $X$).
\end{lem}

\begin{prop}
\label{prop:existence}
For each $n \geq 2$ and $2 \leq d,l \leq n$ such that $l \leq d$ and $d + l \leq n+2$    there is an $n+1$-dimensional projective variety $X_{n+1}$ for which:
\begin{itemize}
\item $X_{n+1}$ is the projective cone over a smooth projective variety $Y_{n}$ with vertex $P$,
\item $X_{n+1}$ is $S_{d}$ and $\depth \sO_{X_{n+1},P} = d$,
\item $\omega_{X_{n+1}}$ is $S_l$ and $\depth \omega_{X_{n+1},P} = l$,
\item $\omega_{X_{n+1}}$ is a $\bQ$-line bundle.
\end{itemize}
\end{prop}

\begin{proof}
Take first two Calabi-Yau hypersurfaces $Z$ and $W$ of dimension $d-1$ and $n+1-l$, respectively.  Let $Y:=Y_{n}$ be a general high enough degree complete intersection of codimension $d-l$ in $Z \times W$. Notice, that $d-l \geq 0$ by assumption. Finally, let $X_{n+1}$ be the projective cone over $Y$ polarized by $\sO_Y(1) := \sO_{Z \times W}(p)|_Y$ for some $p \gg 0$ (after fixing $ Y$). Here $\sO_{Z \times W}(1)$ is the very ample line bundle on $Z \times W$ coming from its projective embedding.

The K\"unneth isomorphism yields the following.
\begin{equation*}
H^q(Z \times W, \sO_{Z \times W}) \cong \bigoplus_{r=0}^q H^r(Z,\sO_Z) \otimes H^{q-r}(W, \sO_{W}) .
\end{equation*}
Since $Z$ and $W$ are Calabi-Yau hypersurfaces of dimension $d-1$ and $n+1-l$, respectively, the following holds for their cohomology table.
\begin{equation*}
 H^q(Z, \sO_{Z }) \neq 0 \Leftrightarrow q=0 \textrm{ or } d-1 
 \end{equation*}
 \begin{equation*}
 H^s(W, \sO_{W}) \neq 0 \Leftrightarrow s=0 \textrm{ or } n+1-l.
\end{equation*}
Hence 
\begin{equation*}
 H^q(Z \times E, \sO_{Z \times E}) \neq 0 \Leftrightarrow q=0, d-1, n+1-l \textrm{ or } n-l +d   .
\end{equation*}
Using Lemma \ref{lem:complete_intersection} yields
\begin{equation}
\label{eq:coho_table}
\textrm{ for } 0<q<n : \qquad H^q(Y, \sO_Y) \neq 0 \Leftrightarrow q= d-1 \textrm{ or } n +1-l .
\end{equation}
Since $p \gg 0$, also:
\begin{equation*}
H^q(Y, \sO_{Y}(r)) =0  \textrm{ for every } r \textrm{ and }  0 < q < n .
\end{equation*}
Then by Lemma \ref{lem:S_d_cones} using that $d-1 \leq n+1-l$ by assumption, $X_{n+1}$ is $S_{d}$ and $\depth \sO_{X_{n+1},P}=d$ ($X_{n+1}$ is $S_2$ at the vertex, because $p \gg 0$ and hence $Y$ is projectively normal).

Serre duality implies that 
\begin{equation*}
H^q(Y, \omega_Y) \cong (H^{n-q}(Y, \sO_Y) )^* .
\end{equation*}
So, by \eqref{eq:coho_table},
\begin{equation*}
 \textrm{ for } 0<q<n : \qquad H^q(Y, \omega_Y) \neq 0  \Leftrightarrow q = l-1 \textrm{ or } n+1-d .
\end{equation*}
Since $X_{n+1}$ is an affine bundle over $Y$, $\omega_{X_{n+1}}$ is isomorphic to the pullback of $\omega_Y$ outside of the vertex. Then by Lemma \ref{lem:S_d_cones} using that $l-1 \leq n+1-d$, $\omega_{X_{n+1}}$ is $S_l$ and $\depth \omega_{X_{n+1},P}= l$ ($\omega_{X_{n+1}}$ is always $S_2$ by Proposition \ref{prop:base_change_complex}.\ref{itm:base_change_complex:omega_S2}). 

We are left to show, that $\omega_{X_{n+1}}$ is $\bQ$-Cartier. By Lemma \ref{lem:complete_intersection_adjunction},
\begin{equation*}
\omega_Y^{\otimes p} \cong (\omega_{Z \times E}(m)|_Y)^{\otimes p} \cong( \sO_{Z \times E}(m)|_Y)^{\otimes p} \cong \sO_Y(m).
\end{equation*}
That is, $\omega_Y^{\otimes p}$ is an integer multiple of the polarization of $Y$ used at the construction of  $X_{n+1}$. Hence, \cite[Exercise 3.5]{HCD_KSJ_RA} concludes the proof.
\end{proof}

\section{Construction of families without the base change property}
\label{sec:construction_of_families}

In this section we present the proof of Theorem \ref{thm:main1}. The following lemma contains the key argument of Theorem \ref{thm:main1}. It is also used in the proofs of Proposition \ref{prop:main} and Theorem \ref{thm:main3}.

\begin{lem}
\label{lem:restriction_map_pencil}
Let $f : \sH \to B = \bP^1$ be a flat pencil of hypersurfaces of a quasi-projective, equidimensional scheme $X$, such that $\sI_{\sH , X \times B}$ is a line bundle and $\sH$ and the closed fibers of $f$ are $S_2$ and $G_1$. Then 
\begin{enumerate}                                                                   \item \label{itm:restriction_map_pencil:omega_X_S3} if $\omega_X$ is $S_3$, the restriction map $\omega_{\sH/B}|_{\sH_0} \to \omega_{\sH_0}$ is an isomorphism, 
\item \label{itm:restriction_map_pencil:omega_X_not_S3} if $\depth \omega_{X,P} \not\geq \min\{3, \dim \sO_{X,P}\}$  for some $P \in X$, such that $P \in \sH_0$, but $P \not\in \sH_{\infty}$, then  $\omega_{\sH/B}|_{\sH_0} \not\cong \omega_{\sH_0}$.
\end{enumerate}
\end{lem}

\begin{proof}
Notice that by flatness of $\sH$ and by Lemma \ref{lem:ideal_sheaves_line_bundles_flat}, it does make sense to talk about $\omega_{\sH/B}$. Define $\sX := X \times B$. Then $\sH$ is a hypersurface of $\sX$. By Proposition \ref{prop:restriction_Cartier_divisor}.\ref{itm:restriction_Cartier_divisor:sheaf} there is a homomorphism $\omega_{\sX/B}(\sH)|_{\sH} \to \omega_{\sH/B}$, which is isomorphism in codimension one, over the Gorenstein locus of $\sH$. Fix this homomorphism for the course of the proof. 

Now, we show point \eqref{itm:restriction_map_pencil:omega_X_S3}.   If $\omega_X$ is $S_3$, then so is $\omega_{\sX/B} \cong p_1^* \omega_X$ by  Lemma \ref{lem:S_d_pullback}. Hence, by Fact \ref{fact:S_d_hyperplane}.\ref{itm:S_d_hyperplane:second},  $\omega_{\sX/B}(\sH)|_{\sH}$ is $S_2$. Then, since $\omega_{\sH/B}$ is $S_2$ by Proposition \ref{prop:base_change_complex}.\ref{itm:base_change_complex:relative_omega_S2}, $\omega_{\sX/B}(\sH)|_{\sH} \to \omega_{\sH/B}$ is isomorphism everywhere by \cite[Theorem 1.9 and Theorem 1.12]{HR_GD}. However,  for every $P \in \sX_0$,
\begin{equation}
\label{eq:restriction_map_pencil:S_4}
 \depth \omega_{\sX/B, P} = 
\underbrace{\depth \omega_{X, p_1(P)}  + 1}_{\textrm{Fact \ref{fact:S_d_pullback}, applied to $\omega_{\sX/B} \cong p_1^* \omega_X$}}  
\geq \underbrace{ \min \{ 3, \dim \sO_{X,p_1(P)} \} + 1}_{\omega_X \textrm{ is } S_3} 
= \min \{ 4, \dim \sO_{\sX,P} \} .
\end{equation}
But then, for every $P \in \sH_0$,
\begin{equation*}
\depth \omega_{\sH/B,P} 
=
\underbrace{\depth (\omega_{\sX/B}(\sH)|_{\sH})_P }_{\omega_{\sH/B} \cong \omega_{\sX/B}(\sH)|_{\sH}}
\geq
\underbrace{\min \{ 3, \dim \sO_{\sH,P} \}}_{\textrm{Fact \ref{fact:S_d_hyperplane}.\ref{itm:S_d_hyperplane:second} and \eqref{eq:restriction_map_pencil:S_4}}},
\end{equation*}
 which implies point \eqref{itm:restriction_map_pencil:omega_X_S3} by Proposition \ref{prop:base_change_complex}.\ref{itm:base_change_complex:S2}.

To prove point \eqref{itm:restriction_map_pencil:omega_X_not_S3}, denote by $U$ the  open set $p_1^{-1} (X \setminus (\sH_0 \cap \sH_{\infty} ) ) \subseteq \sX$. This is the set of points, the first coordinates of which are not contained in every element of the pencil $\sH \to B$. By Proposition \ref{prop:base_change_complex}.\ref{itm:base_change_complex:open_set} and \ref{prop:base_change_complex}.\ref{itm:base_change_complex:relative_open_set}, we may replace $\sX$ by $U$, or with other words, $X$ by $X \setminus (\sH_0  \cap \sH_{\infty})$. In particular, then  $\sH_0  \cap \sH_{\infty} = \emptyset$ and $P$ is  an arbitrary point of $\sH_0$, such that 
\begin{equation}
\label{eq:restriction_map_pencil:not_S_3}
 \depth \omega_{X,P} \not\geq \min\{3, \dim \sO_{X,P}\} .
\end{equation}
Then all fibers of the projection $p_1|_{\sH} : \sH \to X$ have dimension zero. So, for every $Q \in \sH$,
\begin{equation*}
\depth \omega_{\sX/B,Q} 
= \underbrace{\depth \omega_{X, p_1(Q)}  + 1}_{\textrm{Fact \ref{fact:S_d_pullback},  applied to $\omega_{\sX/B} \cong p_1^* \omega_X$}}  
\geq \underbrace{ \min \{ 2, \dim \sO_{X,p_1(Q)} \} + 1}_{\textrm{Proposition \ref{prop:base_change_complex}.\ref{itm:base_change_complex:omega_S2}}} 
= \min\{3, \dim \sO_{\sX,Q}\} .
\end{equation*}
Then, repeating the argument of the previous paragraph $\omega_{\sX/B}(\sH)|_{\sH} \cong \omega_{\sH/B}$. Also, at the fixed $P \in \sH_0$, the following computation estimates the depth more precisely.
\begin{equation}
\label{eq:restriction_map_pencil:not_S_4}
\depth \omega_{\sX/B,P} 
%
%
%
= 
\underbrace{\depth \omega_{X,P} + 1}_{\textrm{Fact \ref{fact:S_d_pullback},  applied to $\omega_{\sX/B} \cong p_1^* \omega_X$}}
 \not\geq
\underbrace{\min\{3,\dim \sO_{X,P}\}+1}_{\textrm{\eqref{eq:restriction_map_pencil:not_S_3}}}
=
\min\{4,\dim \sO_{\sX,P} \}
\end{equation}
 However, then 
\begin{equation*}
\depth \omega_{\sH/B,P} 
=
\underbrace{\depth (\omega_{\sX/B}(\sH)|_{\sH})_P }_{\omega_{\sH/B} \cong \omega_{\sX/B}(\sH)|_{\sH}}
\not \geq 	
\underbrace{\min\{3,\dim \sO_{\sH,P} \} }_{\textrm{by Fact \ref{fact:S_d_hyperplane}.\ref{itm:S_d_hyperplane:second}}},
\end{equation*}
which concludes the proof by Proposition \ref{prop:base_change_complex}.\ref{itm:base_change_complex:S2}.

\end{proof}

\begin{rem}
The condition $\sI_{\sH,\sX}$  being a line bundle in Lemma \ref{lem:restriction_map_pencil} might look superfluous for the first sight, since $\sH$ is a hypersurface in $\sX$. However, according to Section \ref{sec:notation}, the latter only means that $\sH$ is the zero locus of some special section of a line bundle. That is, $\sH$ or $\sH_b$ for some $b \in B$ could contain an entire irreducible component of $\sX$ or $\sX_b$, respectively. Then Proposition \ref{prop:restriction_Cartier_divisor} would not apply. Such situations should definitely be avoided.
\end{rem}

The following is the main construction to which Lemma \ref{lem:restriction_map_pencil} is applied in this section.

\begin{const}  
\label{const:main}
Consider a projective cone $X$ over a variety $Y$. Let $P$ be the vertex of $X$. Take two hypersurfaces in $X$. The first one $H$ is a projective cone over a  degree $d$ generic hypersurface $D$ of $Y$. The second one $\widetilde{H}$ is a general degree $d$ hypersurface  of $X$. Denote by $\sH \to B$ the pencil generated by $H$ and $\widetilde{H}$ (for which $H = \sH_0$ and $\widetilde{H} = \sH_{\infty}$). Throughout the paper we allow ourselves to replace this family by its restriction to a small enough open neighborhood of $0 \in B$. Furthermore, when we compute stable reduction in Section \ref{sec:stable_reduction}, we will assume that $d \gg 0$. 
\end{const}

\begin{lem}
\label{lem:construction_main_consequences}
In the situation of Construction \ref{const:main}, if $X$ is $S_3$ and $Y$ is $R_1$, then 
\begin{enumerate}
 \item \label{itm:construction_main_consequences:normal} $\sH$ and the closed fibers of $f$ are normal varieties,
 \item \label{itm:construction_main_consequences:line_bundle} $\sI_{\sH,X \times B}$ is a line bundle,
\item \label{itm:construction_main_consequences:flat} $f$ is flat.
\end{enumerate}
\end{lem}

\begin{proof}
We use the notation $\sX:= X \times B$. Since $Y$ is a variety (i.e., integral), so are $D$, $X$, $\sX$, $\sH_0$ and $\sH_{\infty}$. By the definition of a pencil $\sH$ is defined by a single non-zero equation locally on $\sX$. So, since $\sX$ is integral, point \eqref{itm:construction_main_consequences:line_bundle} follows. Similarly, also $\sH_b$ for every $b \in B$ is defined locally by a single non-zero equation locally. Hence by integrality of $X$, $\sI_{\sH_b,\sX_b}$ is also a line bundle for every $b \in B$. Thus, Lemma \ref{lem:ideal_sheaves_line_bundles_flat} yields point \eqref{itm:construction_main_consequences:flat}.

To prove point \eqref{itm:construction_main_consequences:normal}, note that $\sX$ is $S_3$ by Lemma \ref{lem:S_d_pullback} and by the assumption of the lemma. So, by Fact \ref{fact:S_d_hyperplane}, $\sH$ and the closed fibers of $\sH$ are $S_2$ (Remember, in Construction \ref{const:main} we allowed ourselves to shrink $B$ around $0 \in B$). Since $D$ is general and $Y$ is $R_1$, $D$ is $R_1$ as well by Bertini's theorem (c.f.,  \cite[Theorem 17.16]{HJ_AG}). Therefore, so is $H$. Then, by possibly shrinking $B$, each closed fiber of $\sH$ is $R_1$. Thus all closed fibers of $\sH$, and $\sH$ itself are normal. 
\end{proof}

\begin{thm}
\label{thm:family}
In the situation of Construction \ref{const:main}, if $\dim X \geq 3$, $X$ is $S_3$, $Y$ is  $R_1$, and $\depth \omega_{X,P}=2$, then
\begin{equation}
\label{eq:not_compatible}
\omega_{\sH/B} |_{\sH_0} \not\cong \omega_{\sH_0}.
\end{equation}
In addition:
\begin{enumerate}
\item \label{itm:Q_Gor} if $\omega_X$ is a $\bQ$-line bundle, then $\omega_{\sH/B}$ is a $\bQ$-line bundle. In particular then $\omega_{\sH_b}$ is a $\bQ$-line bundle for all $b \in B$,
\item \label{itm:S_d} if $X$ is $S_d$ and $\depth \sO_{X,P} = d$, then  $\sH_b$ is $S_{d-1}$ , for all $b \in B$, and $\depth \sO_{\sH_0,P} = d-1$.
\end{enumerate}

\end{thm}

\begin{proof}
By Lemma \ref{lem:construction_main_consequences}, we may apply Lemma \ref{lem:restriction_map_pencil}.\ref{itm:restriction_map_pencil:omega_X_not_S3}, to obtain the main statement of the theorem. 

To prove addendum \eqref{itm:Q_Gor}, note that the normality  of $\sH$ and $\sH_b$ for every $b \in B$, \cite[Theorem 1.12]{HR_GD} and  Proposition \ref{prop:restriction_Cartier_divisor} imply that
\begin{equation}
\label{eq:restr_omega_power} 
\omega_{\sH_b}^{[n]} \cong (\omega_{\sX_b}(\sH_b)|_{\sH_b})^{[n]} \textrm{ for any $b \in B$, and } \omega_{\sH / B}^{[n]} \cong (\omega_{\sX/B}(\sH)|_{\sH})^{[n]} 
\end{equation}
for all $n \in \bZ$. Hence if $\omega_{X}$ is a $\bQ$-line bundle, then  \eqref{eq:restr_omega_power} implies that so is $\omega_{\sH/B}$ and $\omega_{\sH_b}$ for all $b \in B$. To prove \eqref{itm:S_d} we use Fact \ref{fact:S_d_hyperplane} once again.
\end{proof}

\begin{cor} [ (= Theorem \ref{thm:main1}) ]
\label{cor:result1}
For each $n \geq 3$ and $n>j \geq 2$ there is a flat family $\sH \to B$ of  $S_{j}$ (but not $S_{j+1}$), normal varieties of dimension $n$ over some open set $B \subseteq \bP^1$, with $\omega_{\sH/B}$ a $\bQ$-line bundle, such that
\begin{equation*}
\left. \omega_{\sH/B} \right|_{\sH_0} \not\cong \omega_{\sH_0}.
\end{equation*}
Moreover, the general fiber of $\sH$ can be chosen to be smooth and the central fiber to have only one singular point.
\end{cor}

\begin{proof}
 It follows by combining Proposition \ref{prop:existence} (setting $d=j+1$ and $l=2$), Construction \ref{const:main} and Theorem \ref{thm:family}.
\end{proof}

\section{Degenerations and Serre's condition}
\label{sec:Cohen-Macaulay}

We turn to proving the statements  relating Serre's condition $S_d$ to degenerations of flat families. The first half of the section is devoted to the following statement.

\begin{thm} [ (= Theorem \ref{thm:main2}) ]
\label{thm:central_fiber_S_n-1}
If $f : \sH \to B$ is a flat family of schemes of pure dimension $n$ over a smooth  curve, such that a component of the locus
\begin{equation}
\label{eq:central_fiber_S_n:locus}
 \overline{\{ x \in \sH | x \textrm{ is closed, } \depth \sO_{\sH_{f(x)},x} = n-1 \} }
\end{equation}
is contained in the special fiber $\sH_0$, then  the restriction homomorphism $\omega_{\sH/B}|_{\sH_0} \to \omega_{\sH_0}$ is not an isomorphism.
\end{thm}


\begin{rem}
By the restriction homomorphism $\omega_{\sH/B} \to \omega_{\sH_0}$ we mean any homomorphism obtained as in Proposition \ref{prop:base_change_complex}.\ref{itm:base_change_complex:restriction_homomorphism}.
\end{rem}

Theorem \ref{thm:central_fiber_S_n-1} might look technical, but it applies for example to the special case, when the general fiber is Cohen-Macaulay and the central fiber contains at least one closed point with depth $n-1$. This yields the following corollary.

\begin{cor} [ (= Corollary \ref{cor:main}) ]
\label{cor:no_S_n-1}
If $f : \sH \to B$ is  a flat family   of  schemes of pure dimension $n$ such that $\omega_{\sH/B}$ is compatible with base change  and the general fiber of $f$ is Cohen-Macaulay, then the central fiber of $f$ cannot have a closed point $x$, such that $\depth \sO_{\sH_{f(x)},x}=n-1$.
\end{cor}

\begin{prop} [ (= Proposition \ref{prop:main}) ]
\label{prop:sharp}
Corollary \ref{cor:no_S_n-1} is sharp in the sense that $n-1$ cannot be replaced with $i$ for any $2 \leq i<n-1$. 
\end{prop}

\begin{proof}
Fix a $2 \leq i < n-1$. Consider the projective cone $X$ given by Proposition \ref{prop:existence}, setting $d=i+1$ and $l=3$. Use then Construction \ref{const:main} for $X$.  By Lemma \ref{lem:construction_main_consequences}, this yields a flat family $f : \sH \to B$ of normal varieties for which Lemma \ref{lem:restriction_map_pencil}.\ref{itm:restriction_map_pencil:omega_X_S3} applies. That is, the restriction homomorphisms $\omega_{\sH/B}|_{\sH_b} \to \omega_{\sH_b}$ are isomorphisms for every $b \in B$. Finally, since $X$ is  Cohen-Macaulay outside of $P$ and $\depth \sO_{X,P}=i+1$, by Fact \ref{fact:S_d_hyperplane},  $\sH_b$ is Cohen-Macaulay outside of $P$, where  $\depth \sO_{\sH_0,P}=i$.
\end{proof}

We also need the following lemma in the proof of Theorem \ref{thm:central_fiber_S_n-1}.

\begin{lem}
\label{lem:tor}
If $f : \sH \to B$ is a flat morphism of schemes onto a smooth curve, $\sF$ is a coherent $\sO_{\sH}$-module on $\sH$ and $P \in \sH_0$, then 
\begin{enumerate}
 \item \label{itm:tor:1} $\sTor^1_{\sH}(\sF,\sO_{\sH_0})_P \neq 0$ if and only if $\sF$ has an associated component $W$ such that  $P \in W \subseteq \sH_0$,
 \item \label{itm:tor:others}  $\sTor^i_{\sH}(\sF,\sO_{\sH_0})=0$ for $i>1$.
\end{enumerate}
\end{lem}

\begin{proof}
By restricting $B$, we may assume that $\sI_{0,B} \cong \sO_B$. Denote by $s$ a generator of $\sI_{0,B}$ and consider the following exact sequence.
\begin{equation}
\label{eq:fiber}
\xymatrix{
0 \ar[r] & \sO_{\sH} \ar[r]^{\cdot s} & \sO_{\sH} \ar[r] & \sO_{\sH_0} \ar[r] & 0 
}
\end{equation}
Then the long exact sequence of $\sTor_{\sH}^{\bullet}(\sF, \_)$ applied to \eqref{eq:fiber} yields
\begin{equation*}
\xymatrix{
\sTor^1_{\sH}(\sF,\sO_{\sH}) = 0 \ar[r] & \sTor^1_{\sH}(\sF,\sO_{\sH_0}) \ar[r] & \sF \ar[r]^{\cdot s} & \sF .
}
\end{equation*}
Hence $\sTor^1_{\sH}(\sF,\sO_{\sH_0})_P \neq 0$ if and only if $s$ annihilates something in $\sF_P$, if and only if $\sF$ has an associated component $W$ such that $P \in W \subseteq \sH_0$.

Another part of the long exact sequence of $\sTor_{\sH}^{\bullet}(\sF, \_)$ applied to \eqref{eq:fiber} yields the following for $i>1$.
\begin{equation*}
\xymatrix{
\sTor^i_{\sH}(\sF,\sO_{\sH}) = 0 \ar[r] & \sTor^i_{\sH}(\sF,\sO_{\sH_0}) \ar[r] & \sTor^{i-1}_{\sH}(\sF,\sO_{\sH}) =0
}
\end{equation*}
Hence, $\sTor^i_{\sH}(\sF,\sO_{\sH_0})=0$ indeed if $i>1$.
\end{proof}


\begin{proof}[Proof of Theorem \ref{thm:central_fiber_S_n-1}] 
Fix a closed point $x \in \sH_0$ with  $\depth \sO_{\sH_0,x} = n-1$,  contained in a component $W \subseteq \sH_0$ of the locus \eqref{eq:central_fiber_S_n:locus}. Notice that the locus \eqref{eq:central_fiber_S_n:locus} is $\supp (h^{-(n-1)}( \omega_{\sH/B}^{\bullet}))$ by Proposition \ref{prop:base_change_complex}.\ref{itm:base_change_complex:relative_Sd_at_closed_point}, and hence $W$ is also an associated component of $h^{-(n-1)}(\omega_{X/B}^{\bullet})$. Consider an open neighborhood of $x$, where every closed point has depth at least $n-1$. Replacing $\sH$ by this neighborhood all assumptions of the theorem stay valid, and moreover we may assume that every closed point of $\sH$ has depth at least $n-1$. In particular, then 
\begin{equation}
\label{eq:non-zero2}
h^i(\omega_{\sH/B}^{\bullet}) \neq 0 \Leftrightarrow i=-n \textrm{ or } -(n-1) .
\end{equation}
Define $\sE:= h^{-(n-1)}(\omega_{\sH/B}^{\bullet})$. By \eqref{eq:non-zero2}, there is an exact triangle
\begin{equation}
\label{eq:dualizing}
\xymatrix{
\omega_{\sH/B}[n] \ar[r] & \omega_{\sH/B}^{\bullet} \ar[r] & \sE[n-1] \ar[r]^-{+1} & .
}
\end{equation}
Applying $\_ \otimes_L \sO_{\sH_0}$  to \eqref{eq:dualizing} and then considering the long exact sequence of cohomology sheaves yields 
\begin{equation}
\label{eq:central_fiber_S_n-1:derived_restriction}
\xymatrix{
  h^{-n-1}(\sE[n-1] \otimes_L \sO_{\sH_0}) \ar[r] & h^{-n} (\omega_{\sH/B}[n] \otimes_L \sO_{\sH_0}) \ar[r] & h^{-n}( \omega_{\sH/B}^{\bullet} \otimes_L \sO_{\sH_0}) \ar[r] & \\ 
\ar[r] & h^{-n}(\sE[n-1] \otimes_L \sO_{\sH_0}) \ar[r] & h^{-n+1} (\omega_{\sH/B}[n] \otimes_L \sO_{\sH_0})  ,
}
\end{equation}
where 
\begin{itemize}
\item $h^{-n-1}(\sE[n-1] \otimes_L \sO_{\sH_0}) \cong \sTor^2_{\sH}(\sE,\sO_{\sH_0}) = 0$ by Lemma \ref{lem:tor},
\item $h^{-n} (\omega_{\sH/B}[n] \otimes_L \sO_{\sH_0}) \cong \omega_{\sH/B}|_{\sH_0}$,
\item $h^{-n}( \omega_{\sH/B}^{\bullet} \otimes_L \sO_{\sH_0}) \cong h^{-n}( \omega_{\sH_0}^{\bullet}) \cong \omega_{\sH_0}$ by Proposition \ref{prop:base_change_complex}.\ref{itm:base_change_complex:restriction_isomorphism},
\item $h^{-n}(\sE[n-1] \otimes_L \sO_{\sH_0}) \cong \sTor^1_{\sH}(\sE,\sO_{\sH_0})$ and
\item $h^{-n+1} (\omega_{\sH/B}[n] \otimes_L \sO_{\sH_0}) \cong h^{1} (\omega_{\sH/B} \otimes_L \sO_{\sH_0})=0$ since $\_ \otimes_L \sO_{\sH_0}$ is left derived functor, so $\omega_{\sH/B} \otimes_L \sO_{\sH_0}$ is supported in negative cohomological degrees.
\end{itemize}
Therefore, \eqref{eq:central_fiber_S_n-1:derived_restriction} is isomorphic to the following exact sequence.
\begin{equation*}
\xymatrix{
0 \ar[r] & \omega_{\sH/B}|_{\sH_0} \ar[r] &  \omega_{\sH_0} \ar[r] & \sTor^1_{\sH}(\sE,\sO_{\sH_0}) \ar[r] & 0 
}
\end{equation*}
Since  $\sE$ has an associated component through $x$  contained in $\sH_0$, $\sTor^1_{\sH}(\sE,\sO_{\sH_0})_x \neq 0$ by Lemma \ref{lem:tor}, which concludes our proof.
\end{proof}

Having finished the proof of Theorem \ref{thm:central_fiber_S_n-1}, the rest of the section is devoted to the following consequence. See Section \ref{sec:introduction} for motivation on Theorem \ref{thm:structure_canonical_Serre}.

\begin{thm}[(=Theorem \ref{thm:main3})]
\label{thm:structure_canonical_Serre}
If $X$ is an  $S_3, G_2$ scheme of pure dimension $n$, which has a closed point with depth $n-1$, then $\omega_X$ is not $S_3$.
\end{thm}

\begin{proof}
Since the statement of the theorem is local, we may assume $X$ is affine and hence quasi-projective. Restricting to a sufficiently small neighborhood of a point with depth $n-1$, all assumptions of the theorem stay valid and we may assume that all closed points of $X$ have depth at least $n-1$. We use the notation $\sX:= X \times B$. Let $X = \bigcup_{i=1}^r X_i$ be the decomposition into irreducible components. 

Consider a pencil $f : \sH \to B= \bP^1$ of hypersurfaces of $X$ such that  
\begin{enumerate}
 \item $\sH_0$ contains the entire non Gorenstein locus,
  \item \label{itm:structure_canonical_Serre:intersection_components} $\emptyset \neq \sH_0 \cap X_j \neq X_j$ for every $ 1 \leq j \leq r $,
  \item $\sH_{\infty}$ is a general hypersurface.
\end{enumerate}
In particular then, 
\begin{equation}
\label{eq:structure_canonical_Serre:intersection}
(\sH_0 \setminus \sH_{\infty}) \cap X_j \neq \emptyset \quad   \textrm{ for every } \quad   1 \leq j \leq r .
\end{equation}
By definition of the pencil, if $P \in \sH_0 \setminus \sH_{\infty}$, then $P \notin \sH_b$ for any $b \neq 0$. Hence assumption \eqref{itm:structure_canonical_Serre:intersection_components} and \eqref{eq:structure_canonical_Serre:intersection} imply that for all $b \in B$, there is a point of $X_j$ not contained in $\sH_b$.
Note now, that since $X$ is $S_1$, by  Lemma \ref{lem:S_d_pullback}, so is $\sX$. In particular, then all associated points of $X$ and $\sX$ are general points of components. So, since none of $\sH_b$ contains any of the $X_j$, $\sI_{\sH,\sX}$ and $\sI_{\sH_b,\sX_b}$ for every $b \in B$ have non-zero divisor local generators and hence are line bundles. Then $\sH$ is flat over $B$ by Lemma \ref{lem:ideal_sheaves_line_bundles_flat}.

Define the following loci
\begin{equation*}
 Z:= \overline{\{ x \in X | x \textrm{ is closed, } \depth \sO_{X,x} = n-1 \} }
\end{equation*}
\begin{equation*}
 W:= \overline{\{ x \in X | x \textrm{ is closed, } \depth \sO_{\sH_{f(x)},x} = n-2 \} } 
\end{equation*}
By construction and by Fact \ref{fact:S_d_hyperplane}, $W_0 = Z$ and $W = (p^{-1} Z)_{\red}$, where $p : \sH \to X$ is the natural projection. Let $Z'$ be an irreducible component of $Z$ of the highest dimension. By the choice of $\sH_0$ and $\sH_{\infty}$, $Z' \subseteq \sH_0$, and $Z' \not\subseteq \sH_{\infty}$. Furthermore, $\sH_{\infty}$ does not contain any of the irreducible components of $Z$. Hence, the general fiber of the map $W \to B$ will have dimension at most $\dim Z'-1$. So,  $W$ has dimension  $\dim Z'$. Hence,  $Z' \subseteq W_0$ is an irreducible component of $W$. In particular, by Theorem \eqref{thm:central_fiber_S_n-1}, the restriction morphism $\omega_{\sH/B}|_{\sH_0} \to \omega_{\sH_0}$ is not an isomorphism.

On the other hand assume that $\omega_X$ is $S_3$. Since $X$ is $G_2$,  $\sH$ and $\sH_b$ for every $b \in B$ are $G_1$ . In fact,  $\sH$, and $\sH_b$ for a general $b \in B$ are $G_2$ also, but for $\sH_0$ only $G_1$ can be guaranteed. Also, $X$ is $S_3$ by assumption and $\sX$ is $S_3$ because of Lemma \ref{lem:S_d_pullback}. Then $\sH$ and $\sH_b$ for every $b \in B$ are $S_2$  by Fact \ref{fact:S_d_hyperplane}. That is, we may apply Lemma \ref{lem:restriction_map_pencil}.\ref{itm:restriction_map_pencil:omega_X_S3}, which states that the restriction homomorphism $\omega_{\sH/B}|_{\sH_0} \to \omega_{\sH_0}$ is  an isomorphism. This is a contradiction, hence $\omega_X$ cannot be $S_3$.
\end{proof}

\section{Stable reduction}
\label{sec:stable_reduction}

In Construction \ref{const:main}, although the general fiber of $\sH \to B$ has mild, i.e., log canonical, singularities, $\sH_0$ is very singular. The failure of base change for $\omega_{\sH/B}$ implies that by \cite[Theorem 7.9]{KJ_KS_LCS} $\sH_0$ is not Du Bois. By \cite[Theorem 1.4]{KJ_KS_LCS}, it is also not log canonical.  In this section, we compute the stable limit of $\sH \to B$. It is the limit at $0$ of some stable family $\sH' \to \widetilde{B}$. This family has two important properties. First,  $\sH \times_{B} \widetilde{B}|_{\widetilde{B} \setminus \{0\}} \cong \sH'|_{\widetilde{B} \setminus \{0\}}$ for a finite cover $\phi : (\widetilde{B},0) \to (B,0)$ totally ramified at $0$.  Second,  $(\sH')_0$ is log canonical, and hence by \cite[Theorems 1.4 and 7.9]{KJ_KS_LCS},  $\omega_{\sH'/\widetilde{B}}$ commutes with base change. So, $(\sH')_0$ is the ``right'' limit of $\sH$, and the incompatibility of Theorem \ref{thm:main1} can be thought of as a consequence of using the wrong limit in Construction \ref{const:main}. '

\begin{prop}
\label{prop:rel_log_canonical}
Assuming that $Y$ is smooth, the stable limit of  Construction \ref{const:main} is the $d$-fold cyclic cover of $Y$ ramified exactly over $D$, with eigen-line bundles $\sO_Y(-i)$ for $i=0,\dots,d-1$.
\end{prop}

\begin{proof}
First, shrink  $B$ if necessary so that $ \infty \notin B$ and that every fiber apart from $\sH_0$ is log canonical. Notice that this is possible because the general fiber of $\sH$ is smooth by Bertini's theorem. Also, since we assumed that $d \gg 0$, the family $\sH \to B$ has canonically polarized fibers and hence is stable over $B^* := B \setminus \{0\}$. Define $\sX := X \times B$.

The closed embedding $Y \subseteq \bP^{N-1}$ induces a natural closed embedding $X \subseteq \bP^N$. This yields very ample line bundles $\sO_{\bP^N}(1)$ and $\sO_X(1)$. Then, $\sH$ is the zero locus of a section $f_0 + t f_{\infty}$ of $\sO_{\sX}(d) := p_1^* \sO_X(d)$ for some $f_0, f_{\infty} \in H^0(\bP^N,\sO_{\bP^N}(d))$, as explained in Section \ref{sec:notation}. 

Choose a basis $z_0,\dots,z_N$ of $H^0(\bP^N,\sO_{\bP^N}(1))$, such that $z_0,\dots, z_{N-1}$ form a basis of $H^0(\bP^{N-1},\sO_{\bP^{N-1}}(1))$. Then $f_0$ and $f_{\infty}$ correspond to degree $d$ homogeneous polynomials in variables $z_0,\dots,z_{N-1}$ and $z_0,\dots,z_N$, respectively.  Furthermore, the fact that $P \notin \sH_{\infty}$ implies that the coefficient of $z_N^d$ in $f_{\infty}$ is non-zero, say $c$.

Let $\phi: (\widetilde{B},0) \to (B,0)$ be the degree $d$ cyclic cover branched only at $0$, where it is totally ramified, and let $s$ be a local parameter of $\widetilde{B}$ at $0$, such that $s^d=t$. Consider the subscheme $\sH' \subseteq \sX \times_{B} \widetilde{B} =: \sX_{\phi}$ defined by 
\begin{equation*}
f_0(z_0,\dots,z_{N-1}) + s^d f_{\infty} (z_0,\dots, z_{N-1}, \frac{1}{s} z_N) \in H^0(\sX_{\phi}, \sO_{\sX_{\phi}}(d)) ,
\end{equation*}
where $\sO_{\sX_{\phi}}(d)$ is the pullback of $\sO_{\sX}(d)$ to $\sX_{\phi}$.

By the uniqueness of stable limit, $\sH'$ is a stable reduction of $\sH$ (i.e., a stable family isomorphic generically to the pullback of $\sH$), if 
\begin{enumerate}
 \item \label{itm:rel_log_canonical:proof2:central_fiber} $(\sH')_0$ is a canonically polarized manifold, and 
 \item \label{itm:rel_log_canonical:proof2:general_fiber} $\sH'|_{\widetilde{B}^*} \cong \sH_{\phi}|_{\widetilde{B}^*}$, where $\sH_{\phi}:= \sH \times_B \widetilde{B}$ and $\widetilde{B}^*:=\widetilde{B} \setminus \{0 \}$.
 \end{enumerate}
To prove point \eqref{itm:rel_log_canonical:proof2:central_fiber}, notice that $(\sH')_0$ is defined by the zero locus of $s$ on $\sH'$ or equivalently by the zero locus of the following section of $\sO_{X}(d)$ on $X$.
\begin{equation*}
f_0(z_0,\dots,z_{N-1}) + c z_N^d
\end{equation*}
Hence it is the cyclic cover of $Y$ of degree $d$ branched along $D$ with eigensheaves $\sO_Y(-i)$ for $0 \leq i \leq d-1$. So first, it is smooth by \cite[Lemma 2.51]{KJ_MS_BG}. Second, since $(\sH')_0$ is contained in the smooth part of $X$, $\omega_{(\sH')_0} \cong \omega_X(d)|_{(\sH')_0}$ by Proposition \ref{prop:restriction_Cartier_divisor} and is a line bundle. So, since $d \gg 0$, $(\sH')_0$ is a canonically polarized manifold indeed. 

To prove point \eqref{itm:rel_log_canonical:proof2:general_fiber}, notice that the equation of $\sH_{\phi}$ in $\sX_{\phi}$ is 
\begin{equation*}
f_0(z_0,\dots,z_{N-1}) + s^d f_{\infty} (z_0,\dots, z_{N-1}, z_N) \in H^0(\sX_{\phi}, \sO_{\sX_{\phi}}(d))
\end{equation*}
Hence, $\sH_{\phi}|_{\widetilde{B}^*} \cong \sH'|_{\widetilde{B}^*}$ via the isomorphism induced by the following automorphism of $\bP^N \times \widetilde{B}^*$.
\begin{equation*}
(z_0, \dots, z_{N+1}, z_N ) \mapsto (z_0,\dots,z_{N-1}, s z_N) 
\end{equation*}

We proved both points \eqref{itm:rel_log_canonical:proof2:central_fiber} and \eqref{itm:rel_log_canonical:proof2:general_fiber}. Consequently, $\sH'$ is a stable reduction of $\sH$ indeed. Through the course of the proof of point \eqref{itm:rel_log_canonical:proof2:central_fiber}, we also proved that $(\sH')_0$ is indeed the cyclic cover described in the statement of the proposition. 
\end{proof}


\bibliographystyle{skalpha}
\bibliography{include}

\end{document}